\documentclass[12pt,a4paper]{article}

\usepackage{amssymb}
\usepackage{amsmath, amsthm}
\usepackage{arydshln}
\usepackage{epsfig}
\usepackage{setspace}
\usepackage{comment}
\usepackage{geometry}
\usepackage{hyperref}

\def\RR{\mathbb{R}}
\def\CC{\mathbb{C}}
\def\ZZ{\mathbb{Z}}
\def\QQ{\mathbb{Q}}
\def\FF{\mathbb{F}}
\def\KK{\mathbb{K}}
\def\11{\mathbf{1}}

\numberwithin{equation}{section}

\newcommand{\rank}{\mathop{\rm rank} }

\newcommand{\Conv}{\mathop{\rm Conv} }

\newcommand{\Det}{\mathop{\rm Det} }
\newcommand{\ncrank}{\mathop{\rm nc\mbox{-}rank} }

\newcommand{\proj}{{\rm proj}}

\newcommand{\expvec}{\mathop{\rm vec}}
\newcommand{\vecspan}{\mathop{\rm span}}
\newcommand{\ncdet}{\mathop{\rm Det} }
\newcommand{\mindeg}{\mathop{\rm mindeg} }

\newcommand{\eps}{\varepsilon}
\DeclareMathOperator{\poly}{poly}

\newtheorem{Thm}{Theorem}[section]
\newtheorem{Prop}[Thm]{Proposition}
\newtheorem{Lem}[Thm]{Lemma}

\theoremstyle{definition}
\newtheorem*{Clm}{Claim}

\newtheorem{Rem}[Thm]{Remark}

\usepackage{xcolor}

\usepackage[textsize=scriptsize]{todonotes}

\title{
Algebraic combinatorial optimization\\
on the degree of  determinants \\ of 
noncommutative symbolic matrices}
\author{Hiroshi HIRAI\footnote{Graduate School of Mathematics,
		Nagoya University, Furocho, Chikusaku, Nagoya, 464-8602, Japan.
		\texttt{\footnotesize hirai.hiroshi@math.nagoya-u.ac.jp}},\quad
		Yuni IWAMASA\footnote{Graduate School of Informatics, Kyoto University, Kyoto, 606-8501, Japan.
        \texttt{iwamasa@i.kyoto-u.ac.jp}},\quad
		Taihei OKI\footnote{Department of Mathematical Informatics, Graduate School of Information Science and Technology, The University of Tokyo, Tokyo, 113-8656, Japan.
        \texttt{oki@mist.i.u-tokyo.ac.jp}},\quad
		Tasuku SOMA\footnote{The Institute of Statistical Mathematics, Tokyo, 190-8562, Japan.
        \texttt{soma@ism.ac.jp}}
}

\begin{document}
	\maketitle

\begin{abstract}
We address the computation of the degrees of minors of a noncommutative symbolic matrix of form
\[
	A[c] := \sum_{k=1}^m A_k t^{c_k} x_k, 
\]
where $A_k$ are matrices over a field $\KK$, $x_i$ are noncommutative variables, $c_k$ are integer weights, and $t$ is a commuting variable specifying the degree. This problem extends noncommutative Edmonds' problem (Ivanyos et al. 2017), and can formulate various combinatorial optimization problems. Extending the study by Hirai 2018, and Hirai, Ikeda 2022, we provide novel duality theorems and polyhedral characterization for the maximum degrees of minors of $A[c]$ of all sizes, and develop a strongly polynomial-time algorithm for computing them. 
This algorithm is viewed as a unified algebraization of the classical Hungarian method for bipartite matching and the weight-splitting algorithm for linear matroid intersection. 
As applications, we provide polynomial-time algorithms for weighted fractional linear matroid matching and linear optimization over rank-2 Brascamp-Lieb polytopes.
\end{abstract}

Keywords: Combinatorial optimization, noncommutative Edmonds' problem, nc-rank, Dieudonn\'e determinant, Euclidean building,
fractional linear matroid matching, 
Brascamp-Lieb polytope.\\

MSC classifications: 90C27, 68Q25

\section{Introduction}
The theory of combinatorial optimization, particularly for well-solvable classes of combinatorial optimization problems, was founded by J.~Edmonds through his groundbreaking work in the 1960s and 1970s. 
In addition to the notion of a polynomial-time algorithm itself, 
he invented several important concepts on it, such as good characterization 
($\simeq$ polynomial-time verifiable min-max relation $\simeq$  NP $\cap$ co-NP characterization) via LP-duality and total dual integrality, and a powerful framework of polymatroids/submodular flows. With boosting of the separation-optimization equivalence~\cite{GLS} via the ellipsoid method, 
these LP-based polyhedral methods (a.k.a. {\em polyhedral combinatorics}) 
are now understood as a basic paradigm in combinatorial optimization, 
and were culminated as the three-volume book~\cite{SchrijverBook} by Schrijver.

The present paper addresses a different direction in combinatorial optimization, which may go beyond the above paradigm of ``Polyhedra and Efficiency.''
In his 1967 paper~\cite{Edmonds67}, 
Edmonds also suggested an algebraic generalization of the bipartite matching problem, 
which asks to compute the rank of a matrix $A$ given as the following form
\begin{equation}\label{eqn:A}
	A = \sum_{k=1}^m A_k x_k,
\end{equation}
where $A_k$ are matrices over a field $\KK$,
$x_k$ are variables, and the rank of $A$ is considered
in the rational function field $\KK (x_1,x_2,\ldots,x_k)$.
This problem is sometimes called {\em Edmonds' problem}.
A deterministic polynomial-time algorithm 
for Edmonds' problem is not known, 
and is one of the major open problems in theoretical computer science.
The maximum cardinality bipartite matching problem is just the rank computation of (\ref{eqn:A}) such that each $A_k$ has exactly one nonzero entry. Other representatives of polynomially-solvable combinatorial optimization problems, such as linear matroid intersection, nonbipartite matching, and linear matroid matching, are also formulated as Edmonds' problem for special matrices;  see e.g., \cite{GeelenSurvey, Lovasz89}. 
Although min-max theorems in some of these problems were found/can be deduced via such algebraic formulation (see e.g., \cite[Section 16.2b]{SchrijverBook}), 
it is not as well understood as the polyhedral methods.

{\em Noncommutative Edmonds' problem}, introduced by Ivanyos, Qiao, and Subrahmanyam~\cite{IQS15a}, is bringing about new algebraic understanding on combinatorial optimization.
In this setting, 
$x_k$ are regarded as noncommutative variables
$(x_ix_j \neq x_jx_i)$
and $A$ is regarded as a matrix over the noncommutative polynomial ring. 
As the polynomial ring embeds into the rational function field, 
the noncommutative polynomial ring 
embeds into the most generic skew field of fractions, 
called the {\em free skew field}~\cite{Amitsur, Cohn}.
The rank of $A$ over the free skew field 
is called  the {\em noncommutative rank (nc-rank)} of $A$, 
which is denoted by $\ncrank A$.
The nc-rank is an upper bound on the rank.  
The Fortin-Rautenauer formula~\cite{FortinReutenauer04} says that nc-rank is obtained by an optimization problem over the family of vector subspaces.
For the matrices corresponding to bipartite matching, and more generally 
linear matroid intersection,  
the rank and the nc-rank are equal. Consequently, 
the min-max formulas are deduced from the Fortin-Rautenauer formula. 
This derivation is fundamentally different from the LP-based polyhedral methods. Further, the nc-rank can be computed in polynomial time~\cite{GGOW15,HamadaHirai,IQS15b}: 
One is based on a vector-space generalization (the {\em Wong sequences}) of alternating paths~\cite{IQS15b}, and the other two are based on geodesically convex optimization on nonpositively 
curved manifolds~\cite{GGOW15} and non-manifolds~\cite{HamadaHirai} (symmetric spaces and Euclidean buildings).
They are  beyond Euclidean convex optimization algorithms.

Edmonds' problem for (\ref{eqn:A}) naturally extends to a weighted version: 
For an integer weight vector $c = (c_k)$, we consider the following matrix
\begin{equation}\label{eqn:A[c]}
	A[c] := \sum_{k=1}^m A_k t^{c_k} x_k,
\end{equation}
where $t$ is a new variable.
Then, the computation of the degree $\deg \det A[c]$ of the determinant $\det$ of $A[c]$ with respect to $t$
is an algebraic correspondent of weighted maximization in combinatorial optimization. 
Actually, it captures the weighted maximization of the above-mentioned problems. 
In response to this fact and the above development, 
Hirai~\cite{HH_degdet} formulated a noncommutative version of the deg-det computation, by using the {\em Dieudonn\'e determinant} $\Det$, a determinant concept for matrices over a skew field~\cite{Dieudonne}. 
He established a duality formula for the degree of the Dieudonn\'e determinants, gave a simple pseudo-polynomial time algorithm (the {\bf Deg-Det} algorithm) 
to compute $\deg \Det$, and showed the $(\deg \det = \deg \Det)$ relation 
for linear matroid intersection.  
Also, he noted that the {\bf Deg-Det} algorithm can realize
the greedy algorithm for linear matroid 
and the Hungarian method for bipartite matching.
Furue and Hirai~\cite{FurueHirai} showed that it can also deduce the
{\em weight-splitting algorithm}~\cite{Frank1981} for linear matroid intersection, with a new matrix implementation.
Further, Hirai and Ikeda~\cite{HiraiIkeda} showed that $\deg \Det A[c]$ can be computed in strongly polynomial time.  
Their algorithm is based on cost-scaling of $c$ and 
simultaneous Diophantine approximation (the Frank-Tardos method)~\cite{FrankTardos1987}, both of which are standard techniques in combinatorial optimization. 
They also gave a polyhedral interpretation of the deg-Det computation of $A[c]$: 
It can be viewed as linear optimization over an integral polytope determined by $A$. 
This provides a  
broad class of polynomially optimizable integral polytopes, 
including the perfect matching polytope of a bipartite graph and the common base polytope of two linear matroids. It also  
suggests a polyhedral approach to Edmonds' problem.

The present paper pursues this emerging direction of research, 
which might be called {\em algebraic combinatorial optimization}. 
The organization and results of this paper are outlined as follows.
In Section~\ref{sec:nc-rank/deg-Det}, we summarize previous results on nc-rank/deg-Det and their connections to combinatorial optimization.
Particularly, we introduce the {\bf Deg-Det} algorithm, 
which is a simple and flexible iterative algorithm using the nc-rank computation as a subroutine. It is an algorithmic as well as theoretical basis in the subsequent sections. 

From Section~\ref{sec:main}, 
we present the contributions of this paper. 
Extending the previous study~\cite{HH_degdet,HiraiIkeda} on the degree of determinants, 
we consider the maximum degrees of subdeterminants for aiming to capture cardinality-(un)constraint maximization.
We first deal with linear symbolic rational matrices $B = \sum_{k}B_k x_k$, where $B_k$ is a matrix over the rational function field $\KK(t)$.  
Let $\varDelta_\ell(B)$ 
denote the maximum degree of Dieudonn\'e subdeterminants of $B$ of size $\ell$, 
and let $\varDelta_{\rm max}(B) := \max_{\ell} \varDelta_\ell(B)$.
We establish in Theorem~\ref{thm:duality_B} duality formulas for $\varDelta_\ell(B)$ and $\varDelta_{\rm max}(B)$.
The proof is based on the {\bf Deg-Det} algorithm, 
which yields an algorithm ({\bf Deg-SubDet}) 
to compute $\varDelta_{\ell}(B)$ for all $\ell$ at once.
Then, we focus on linear symbolic monomial matrices $A[c]$ in (\ref{eqn:A[c]}), and establish sharpened duality formulas 
for $\varDelta_{\ell}(A[c])$ and $\varDelta_{\rm max}(A[c])$, which link with good characterization and generalize Iwamasa's duality theorem~\cite{Iwamasa_degdet} for $2 \times 2$-partitioned matrices. 
Extending the notion of {\em nc-Newton polytopes} in \cite{HiraiIkeda}, we introduce 
{\em nc-independent set polytopes} ${\cal Q}(A)$ for $A$
and show that $\varDelta_{\rm max}(A[c])$ is interpreted as the optimal value of linear optimization over ${\cal Q}(A)$ with respect to the weight vector $c$. 
This class of polytopes includes the bipartite matching polytope, linear matroid independent set polytope, linear matroid intersection polytope, and 
turns out to include the fractional matching and fractional matroid matching polytopes. 
We provide a strongly polynomial-time specialization of 
the {\bf Deg-Det} algorithm (Hungarian {\bf Deg-Det}).
This algorithm uses neither cost-scaling nor the Frank-Tardos method, and
is viewed as a dual-only-driven algebraic abstraction of the classical Hungarian method. 
With an appropriate primal specification, 
it can realize the original Hungarian method and the 
weight-splitting algorithm for weighted linear matroid intersection. 
Also, it brings a new strongly polynomial-time algorithm 
for the deg-det computation on $2\times 2$-partitioned matrices.
This algorithm is a concise modularized variant of 
Iwamasa's algorithm~\cite{Iwamasa_degdet}, 
where 
the unweighted algorithm in Hirai and Iwamasa~\cite{HiraiIwamasaIPCO} 
is used as a subroutine.

In Section~\ref{sec:applications}, we present applications. 
For the linear matroid matching case, 
the rank and nc-rank are not equal. 
However, Oki and Soma~\cite{OkiSoma} showed that the nc-rank 
is equal to twice the maximum size of fractional matroid matchings.
We show its weighted extension: $\varDelta_{\rm max}(A[c])$ is equal to the maximum weight of fractional matroid matchings, with respect to the weight vector $c$. 
As a consequence, the nc-independent set polytope ${\cal Q}(A)$ 
is the twice expansion of the fractional linear matroid matching polytope. 
Also, our Hungarian {\bf Deg-Det} algorithm brings 
a new strongly polynomial-time algorithm, which is a variation 
of  Gijswijt and Pap~\cite{GP13}, and can care about the bit complexity. 
This result can be applied to the {\em Brascamp-Lieb inequality}~\cite{BrascampLieb}. 
It is known~\cite{BCCT08} that the parameter space of nontrivial BL-inequality forms an interesting rational polytope ({\em BL-polytope}). This polytope is defined by exponentially many inequalities and coincides with a matroid base polytope if the defining matrices are all rank-one~\cite{Barthe1998}.
The computational complexity on BL-polytopes 
is an intriguing open problem in theoretical computer science~\cite{GGOW18}.
Franks, Soma, and Goemans~\cite{FranksSomaGoemans2022} found that the BL-polytope defined by rank-$2$ matrices
is equal to the perfect fractional matroid matching polytope associated with these matrices, 
and showed that the membership problem for the corresponding BL-polytope is in 
NP $\cap$ co-NP. Our result implies that it is indeed in P.

A notable feature of the presented results is that they are described in a building theoretic nature.  
{\em Buildings} (and {\em Euclidean buildings})~\cite{Garrett} 
are a formalism of combinatorics on flags of vector subspaces, and are habitats of 
linearly-representable matroids; see e.g., \cite[Chapter 7]{CoxeterMatroid}.
Particularly, the presented min-max formulas 
for $\varDelta_\ell$ and $\varDelta_{\rm max}$ are described by optimization over Euclidean buildings, and the {\bf Deg-Det} algorithms 
and its time complexity 
are analyzed via arguments having building theory behind.
This will make waves on 
the paradigm of matroid-based algorithm design 
in combinatorial optimization so far. 
Note that any pre-knowledge of the building theory is not required for reading this paper 
as it is described in linear algebra language.

\paragraph{Related work and literature.}
Edmonds' problem is at a crossroad of various fields in mathematical science. In theoretical computer science, it is at the core of deterministic vs.~randomization issue, particularly, 
\emph{polynomial identity testing}, which has been an active research field in theoretical computer science for decades.
If the underlying field is sufficiently large, 
then Edmonds' problem admits a simple randomized polynomial-time algorithm (Lov\'asz~\cite{Lovasz79}). 
However, any deterministic polynomial-time algorithm is not known. Such an algorithm would imply a nontrivial circuit lower bound, which is another long-standing open problem in computational complexity theory~\cite{Kabanets2004}.
Edmonds' problem can also be solved by 
{\em matrix completion} of finding a substitution of $x_k$ 
that attains the maximum rank.
The above randomized polynomial-time algorithm is based on random substitution. Deterministic polynomial-time matrix completions are known for a few classes of matrices; they are related to well-solved combinatorial optimization problems mentioned above; see e.g., \cite{GeelenAlgebraicMatching}.
In engineering, solvability and the degrees of freedom in 
physical dynamical systems are analyzed 
via the rank computation of symbolic matrices, 
where $x_k$ are physical parameters. 
Motivated by this, 
Murota~\cite{MurotaMatrix} 
developed the theory of {\em mixed matrices}, 
a class of symbolic matrices enhanced with matroid algorithms, which is useful for such analysis on dynamical systems (DAEs), electrical networks, and related systems. 
Rigidity of bar-joint frameworks in generic positions 
also falls into the symbolic rank computation.
Several results (e.g., Laman's theorem) as well as long-standing open problems (e.g., 3-dimensional generic rigidity) in rigidity theory are about combinatorial rank characterizations 
of symbolic matrices arising as rigidity matrices; see~\cite{Lovasz89}.

Noncommutative Edmonds' problem further broadens the literature. 
As mentioned, the algorithm (IQS-algorithm) by \cite{IQS15a,IQS15b} for nc-rank is a vector-space generalization of the augmenting path algorithm for bipartite matching. 
This suggests an algebraic generalization, or 
{\em algebraization}, of combinatorial optimization.
It actually led to the study 
of linear-algebraic analogues of several graph quantities by Li et al.~\cite{li2022connections,li2023linearalgebraic}, 
which may share a similar spirit of ours.
Noncommutative Edmonds' problem 
is also closely connected to invariant theory. 
In $\KK = \CC$, 
the nc-full-rank testing of $A$ is equivalent to the {\em null-cone membership} (of asking whether the origin belongs to the orbit closure) of the {\em left-right action} $SL_n(\CC)^2 \ni (g,h) \mapsto (g^{\dagger}A_kh)$.  
The first polynonial-time algorithm of nc-rank, due to Garg, Gurvits, Oliveira, Wigderson~\cite{GGOW15}, is interpreted 
as one minimizing the norm over an orbit closure of this action. 
Subsequently, it turned out to be geodesically convex optimization on nonpositively curved manifolds, {\em symmetric spaces of nonpositive curvature}; 
see e.g., \cite{Allen-Zhu2018,BFGOWW} for further developments. The algorithm by Hamada and Hirai~\cite{HamadaHirai} is based on
discrete convex optimization on a Euclidean building, 
which is the boundary at infinity of the symmetric space; 
see \cite{HiraiHadamard}.
Its objective function has analogous properties ({\em submodularity} and {\em L-convexity}) of {\em discrete convex functions}~\cite{MurotaDCA}, which were earlier developed on $\ZZ^n$ for aiming unified understanding of well-solvable combinatorial optimization problems. 
The min-max theorems presented in this paper are also described by such discrete convex functions, 
{\em L-convex functions on Euclidean buildings}~\cite{HH_degdet}.

\paragraph{Notation.}%\label{sec:preliminary}
Let $\ZZ$, $\QQ$, and $\RR$ denote the sets of integer, rational, 
and real numbers, respectively.
Let $\ZZ_+$, $\QQ_+$, and $\RR_+$ denote their nonnegative subsets.  
%Let $\QQ$ denote the set of rational numbers.
For integers $i,j$ with $i \leq j$, let $[i,j] := \{i,i+1,\ldots,j\}$ denote 
the integer interval from $i$ to $j$.
For a positive integer $n$, let $[n] := [1,n] = \{1,2,\ldots,n\}$.
For $i \in [n]$, let $e_i$ denote the $i$-th unit vector in $\ZZ^n$.
For $X \subseteq [n]$,
let ${\bf 1}_X := \sum_{i \in X} e_i \in \{0,1\}^n$ be its characteristic vector.
We denote ${\bf 1}_{[k]}$ simply by ${\bf 1}_k$, and ${\bf 1}_n$ by ${\bf 1}$.

Let $\KK$ be a field.
Let $\KK[t]$ denote the ring of polynomials
with indeterminate $t$, and let $\KK(t)$ denote the field of rational functions with indeterminate $t$.
The (maximum) degree $\deg p$ of a (nonzero) polynomial $p(t) = \sum_{i=0}^d a_i t^i \in \KK[t]$ with $a_d \neq 0$
is defined as $d$, where $\deg 0 := -\infty$.
Accordingly, the degree $\deg p/q$ of a rational $p/q  \in \KK(t)$ with 
$p \in \KK[t]$, $q \in \KK[t] \setminus \{0\}$ is defined as $\deg p - \deg q$.
By a {\em rational matrix} we mean a matrix over the rational function field $\KK(t)$.

Let $\KK(t)^-$ denote the subring of $\KK(t)$ consisting
of rationals $p/q$ with $\deg p/q \leq 0$.
A {\em proper} matrix is a matrix over $\KK(t)^-$.
A {\em biproper} matrix is a nonsingular proper matrix whose inverse is also proper.
Any rational matrix $B = B(t)$  is expanded as a formal power series
\begin{equation}\label{eqn:formal_power_series}
B = B^{(d)} t^{d} + B^{(d-1)} t^{d-1} + \cdots,
\end{equation}
where $B^{(\ell)}$ is a matrix over $\KK$ and
$d \geq \max_{ij} \deg B_{ij}$.
A biproper matrix is precisely a proper matrix $B$ with 
$B^{(0)} = 0$.
A square matrix $S$
is called {\em upper-unitriangular} (resp.~{\em lower-unitriangular})
if $S$ is an upper-triangular matrix (resp.~lower-triangular matrix) and each diagonal element is $1$.
It is clear that proper unitriangular matrices
are biproper.

For a ring $R$ with unity, let $GL_n(R)$ denote the group
of $n \times n$ matrices $S$ over $R$ with $S$ invertible in $R$.
In this paper, we will deal with
\begin{eqnarray*}
	 GL_n(\KK) & =& \mbox{the group of nonsingular $n \times n$ matrices over $\KK$,} \\
 GL_n(\KK(t)) &=& \mbox{the group of nonsingular $n \times n$ matrices over $\KK(t)$,} \\
	GL_n(\KK(t)^-) &=& \mbox{the group of $n \times n$ biproper matrices} \\
 &=& \mbox{the set of proper matrices $B$ with $B^{(0)} \in GL_n(\KK)$}.
\end{eqnarray*}
%
%\begin{Lem}
%A rational matrix $B = B(t)$ is biproper if and only if
%it is expanded as $B = B^{(0)} + B^{(-1)} t^{-1} + \cdots$, where $B^{(0)}$ is nonsingular.
%\end{Lem}

For integer vector $\alpha \in \ZZ^n$,
let $(t^{\alpha}) \in GL_n(\KK(t))$ denote
the diagonal matrix having diagonals $t^{\alpha_1},t^{\alpha_2},\ldots, t^{\alpha_n}$ in order, that is,
\begin{equation}\label{eqn:t^alpha}
(t^{\alpha}) :=
\left( \begin{array}{cccc}
t^{\alpha_1} &                    &  & \\
& t^{\alpha_2}   &  & \\
&                     & \ddots & \\
&                      &           & t^{\alpha_n}
\end{array}
\right).
\end{equation}
We will allow $\alpha$ to be a rational vector; 
in this case, matrices involving $(t^{\alpha})$ are 
over the field of Puiseux series.

In a vector space $V$, we write $U \leqslant V$ if 
$U$ is a vector subspace of $V$, and $U < V$ if 
 $U \leqslant V$ and $U \neq V$. 

\section{Noncommtative rank and determinant}\label{sec:nc-rank/deg-Det}
\subsection{Noncommutative rank}
Here we summarize basic results of noncommutative rank.

%\subsubsection{Fortin-Rautenauer formula}

\begin{Thm}[\cite{FortinReutenauer04}]\label{thm:FR}
	Let $A$ be a linear symbolic matrix in {\rm (\ref{eqn:A})}.
	Then $\ncrank A$ is equal to the optimal value of the following problem:
	\begin{eqnarray*}
		\mbox{{\rm (FR)}} \quad {\rm Min.} && 2n - r - s  \\
		{\rm s.t.} && \mbox{$S A T$ has an $r \times s$ zero submatrix,} \\
		&& S, T \in GL_n(\KK).
	\end{eqnarray*}
\end{Thm}
By $\rank A = \rank SAT \leq 2n - r - s$, nc-rank is an upper bound of rank:
\[
\rank A \leq \ncrank A.
\]
If $\ncrank A = n$, we say that $A$ is {\em nc-nonsingular}.
\begin{Thm}[\cite{FranksSomaGoemans2022,HamadaHirai,IQS15a,IQS15b}]\label{thm:ncrank_dual}
	An optimal solution $S,T$ in {\rm FR} can be obtained in polynomial time.
\end{Thm}
For $\KK = \QQ$,
the algorithm by Garg et al.~\cite{GGOW15} can compute the optimal value of FR
but cannot obtain an optimal solution $S,T$.
Recently, Franks, Soma, and Goemans~\cite{FranksSomaGoemans2022}
modified this algorithm to obtain an optimal solution $S,T$.
The algorithm by Hamada and Hirai~\cite{HamadaHirai}
obtains optimal $S,T$ even for a small finite field but
does not guarantee polynomial bit complexity in the intermediate computation when $\KK = \QQ$.
Currently, the algorithm by Ivanyos, Qiao, and Subrahmanyam (the {\em IQS-algorithm})~\cite{IQS15a,IQS15b} is the only polynomial-time algorithm to compute an optimal solution $S,T$ with no restriction on $\KK$.

The problem FR is also written as the following vector-subspace optimization (called the {\em Maximum Vanishing Subspace Problem} in \cite{HamadaHirai}):
\begin{eqnarray*}
	\mbox{{\rm (MVSP)}} \quad {\rm Min.} && 2n - \dim U - \dim V  \\
	{\rm s.t.} && A_k(U,V)  = \{0\} \quad (k \in [m]),\\
	&& U,V \leqslant \KK^n,
\end{eqnarray*}
where $A_k$ is regarded as a bilinear form $x,y \mapsto x^{\top}A_ky$.
It should be noted that this problem already appeared in \cite{Lovasz89}.
To see the equivalence between FR and MVSP,
consider nonsingular matrices $S$ and $T$
including bases of $U$ and $V$, respectively.

As observed in \cite{HamadaHirai},
the problem MVSP can be viewed as submodular function minimization on
the modular lattice of vector subspaces.
Indeed, if $U,V$ and $U',V'$ are both feasible to MVSP,
then so are $U \cap U', V + V'$ and $U + U', V \cap V'$.
From $\dim U + \dim U' =  \dim (U \cap U') + \dim (U + U')$,
if $U,V$ and $U',V'$ are both optimal, then
so are  $U \cap U', V + V'$ and $U + U', V \cap V'$.
In particular, there is a unique optimal solution $U,V$
such that $\dim U$ is maximum (and $\dim V$ is minimum).
Such an optimal solution is said to be {\em dominant}.
Accordingly, an optimal solution $S,T$ in FR
is said to be {\em dominant} if $r$ is maximum (and $s$ is minimum).

The IQS-algorithm~\cite{IQS15a,IQS15b} naturally obtains
a dominant optimal solution in (FR).
The algorithm by Franks, Soma, and Goemans~\cite{FranksSomaGoemans2022}
also obtains a dominant optimal solution.

The nc-rank is represented as the ordinary rank of 
an expanded matrix called a {\em blow-up}. 
For a positive integer $d$, 
the {\em $d$-th blow-up} $A^{\{d\}}$ of $A$ is a linear symbolic matrix defined by
\begin{equation}\label{eqn:blow-up}
	A^{\{d\}} := \sum_{k=1}^m A_k \otimes X_k,
\end{equation}
where $\otimes$ denotes the Kronecker product and
$X_k = (x_{k,ij})$ is a $d \times d$ matrix of variable entries $x_{k,ij}$.
We consider the (ordinary) rank of $A^{\{d\}}$ over 
the rational function field $\KK(\{x_{k,ij}\}_{k \in [m],i,j \in [d]})$. 
\begin{Thm}[\cite{IQS15a}]\label{thm:blow-up}
$\displaystyle \ncrank A= \max_{d=1,2,\ldots}\frac{1}{d} \rank A^{\{d\}}$.
\end{Thm}
It is known \cite{DerksenMakam2017} that the maximum in (1) is attained for any $d \ge n-1$.

\subsection{Degree of Dieudonn\'e determinant}\label{subsec:DegDet}
Next we consider the degree of the determinant of the following matrix $B = B(t)$:
\begin{equation}\label{eqn:B}
	B = \sum_{k=1}^m B_k x_k,
\end{equation}
where $x_k$ are variables as above and
$B_k = B_k(t)$ are $n \times n$ matrices 
over rational function field $\KK(t)$ with indeterminate $t$. 
In the commutative case, $\deg \det B$ is the maximum degree, with respect to $t$, 
of rational function $\det B \in \KK(x_1,x_2,\ldots,x_n)(t)$.

The noncommutative formulation is as follows; see \cite{HH_degdet,HiraiIkeda,Oki} for details.
Regard the matrix $B$ as a matrix over the rational function skew field
$\FF(t) =\{p/q \mid p \in \FF[t], q \in \FF[t] \setminus \{0\} \}$,
where $\FF := \KK (\langle x_1,\ldots,x_m \rangle)$ is the free skew field,
$\FF[t]$ is the {\em skew polynomial ring} over $\FF$,
and the degree of an element in $\FF(t)$ is defined similarly.
See \cite[Chapter 7]{Cohn_Algebra} for skew polynomial ring and their fractions.
Any (nonsingular) matrix $B$ over $\FF(t)$ is
written as $B = LDPU$,
where $L$ and $U$ are lower-unitriangular and upper-unitriangular matrices, respectively,
$P$ is a permutation matrix, and $D$ is a diagonal matrix.
This decomposition is called the {\em Bruhat normal form};
see \cite[Theorem 9.2.2]{Cohn_Algebra} and
also Lemma~\ref{lem:Bruhat} below.
The {\em Dieudonn\'e determinant} $\Det B$ of $B$ is defined
as the product of the sign of $P$ and all diagonals of $D$ modulo
the commutator subgroup of the multiplicative group
$\FF(t) \setminus \{0\}$~\cite{Dieudonne}; see \cite[Section 9.2]{Cohn_Algebra}.
We let $\Det B := 0$ if $B$ is singular.
Although
$\Det B$ is no longer an element of $\FF(t)$,
its degree $\deg \Det B \in \ZZ \cup \{-\infty\}$ is well-defined
(as the degree of any commutator is zero).

We distinguish $\det B$ and $\ncdet B$: The former stands 
for the ordinary determinant of $B$ when $B$ is viewed as a matrix over $\KK(x_1,x_2,\ldots,x_n,t)$,
and the latter stands for the Dieudonn\'e determinant of $B$ 
in the above formulation.

As in the ordinary determinant,
it holds $\Det AB = \Det A \Det B$; see~\cite[Chapter 11, Theorem 2.6]{Cohn_Algebra}. Accordingly, it holds
\begin{equation}\label{eqn:degDetAB}
\deg \Det AB = \deg \Det A + \deg \Det B.
\end{equation}
The following property, obviously holding for $\deg \det$, also holds for $\deg \Det$; see \cite{HH_degdet}.
\begin{equation}\label{eqn:degDetB<=0}
\deg B_{ij} \leq 0 \ (i,j \in [n])\ \Rightarrow \ \deg \Det B \leq 0.
\end{equation}
%As mentioned in the introduction, we distinguish
%the ordinary determinant and Dieudonn\'e determinant
%by $\det$ and $\Det$.
%In particular, for a matrix $A[c]$ in (\ref{eqn:A_weighted}),
%we use the following notation:
%\begin{itemize}
%	\item $\deg \det A[c]$ denotes the degree of the determinant of $A[c]$, where $A[c]$ is regarded as a matrix over $\KK(x_1,x_2,\ldots,x_n,t)$.
%	\item $\deg \Det A[c]$ denotes the degree of the Dieudonn\'e determinant of $A[c]$,
%	where $A[c]$ is regarded as a matrix over $\KK(\langle x_1,x_2,\ldots,x_n \rangle)(t)$.
%\end{itemize}

\subsubsection{Duality theorem}
As an extension of the Fortin-Reutenauer theorem (Theorem~\ref{thm:FR}),
Hirai~\cite{HH_degdet} established the following duality theorem for $\deg \Det$:
\begin{Thm}[{\cite{HH_degdet}}]\label{thm:duality_B}
	Let $B = B(t)$ be a matrix in {\rm (\ref{eqn:B})}.
	Then $\deg \Det B$ is equal to the optimal value
	of the following problem (Maximum Vanishing subModule Problem):
	\begin{eqnarray*}
		\mbox{{\rm (MVMP)}} \quad {\rm Min.} && - \deg \det P - \deg \det Q \\
		{\rm s.t.} && \deg (PB_k Q)_{ij} \leq 0 \quad (i,j \in [n], k \in [m]), \\
		&& P,Q \in GL_n (\KK(t)).
	\end{eqnarray*}
\end{Thm}
The weak duality is seen from (\ref{eqn:degDetAB}) and (\ref{eqn:degDetB<=0}) as
\begin{equation}\label{eqn:weakduality_degDet}
0 \geq \deg \Det PBQ = \deg \Det P + \deg \Det B + \deg \Det Q,
\end{equation}
where $\Det P = \det P$ and $\Det Q = \det Q$.
The same weak duality holds for $\deg \det$.
Hence the theorem implies that $\deg \Det$ is an upper bound of $\deg \det$:
\begin{equation*}
	\deg \det B \leq \deg \Det B.
\end{equation*}
We remark that this upper bounding of $\deg \det$ already appeared in~\cite{Murota95_SICOMP}.

%

%
%
%\begin{Lem}
%		The problem (D) is equivalent to
%	\begin{eqnarray*}
	%		\mbox{\rm (D$'$)} \quad {\rm Min.} && - \sum_{i=1}^n \alpha_i- \sum_{j=1}^n \beta_j \\
	%		{\rm s.t.} && \deg ((t^{\alpha})SB_k T(t^{\beta}))_{ij} \leq 0 \quad (i,j \in [n], k \in [m]), \\
	%		&& \alpha, \beta \in \ZZ^n,  S,T \in GL_n (\KK(t)^-).
	%	\end{eqnarray*}
%\end{Lem}
%

%
%For feasible $P,Q$,
%from $0 \geq \deg \det PBQ = \deg \det P + \deg \det Q + \deg \det B$,
%$\deg \Det$ is an upper bound of $\deg \det$:
%\[
%\deg \det B \leq \deg \Det B.
%\]

The feasibility condition of MVMP is rephrased 
as the condition that $PBQ$ is expanded 
as $PBQ = (PBQ)^{(0)} +  (PBQ)^{(-1)}t^{-1} + \cdots.$
The leading term $(PBQ)^{(0)} = \sum_{k} (P B_kQ)^{(0)} x_k$, viewed as a linear symbolic matrix over $\KK$,
plays particularly important roles.
\begin{Lem}[{\cite{HH_degdet}}]\label{lem:optimality}
Let $P,Q$ be a feasible solution for {\rm MVMP}.
	\begin{itemize}
		\item[{\rm (1)}] $P,Q$ is optimal  if and only
		if $\ncrank (PBQ)^{(0)} = n$.
		\item[{\rm (2)}]
		If $\rank (PBQ)^{(0)} = n$, then $\deg \det B = \deg \Det B = - \deg \det P- \deg \det Q$.
	\end{itemize}
\end{Lem}
The second property (2) is useful for showing $\deg \det = \deg \Det$ 
from $\rank = \ncrank$.

As an analogue of Theorem~\ref{thm:blow-up}, it holds:  
\begin{Prop}[{\cite[Lemma 4.5]{HH_degdet}}]\label{prop:blow-up_degDet}
$\displaystyle \deg \Det B = \frac{1}{d}\max_{d=1,2,\ldots} \deg \det B^{\{d\}}$.
\end{Prop}
This property follows from combining Theorem~\ref{thm:blow-up}, Lemma~\ref{lem:optimality}~(2), and $\deg \Det B^{\{d\}} = d \deg \Det B$. Again the maximum is attained by any $d \geq n-1$.
%Notice that the optimality condition (1) does not
%imply a good characterization
%(NP $\cap$ co-NP characterization) for $\det \Det A[c]$,
%since the size of $P, Q$ (e.g., the number of terms)
%may depend on $c_k$ pseudo-polynomially.

%The unboundedness of (D) for $A[c]$ is equivalent to nc-singularity of $A = A[0]$.
%%
%\begin{Lem}\label{lem:nc_nonsingular}
%	$\deg \Det A[c] > - \infty$ if and only
%	if $\ncrank A  = n$.
%\end{Lem}
%

\subsubsection{The Deg-Det algorithm}

The {\bf Deg-Det} algorithm~\cite{HH_degdet} is a conceptually simple algorithm to solve MVMP.
This algorithm uses an algorithm of solving FR as a subroutine, and
is viewed as a simplified version of
Murota's {\em combinatorial relaxation algorithm}~\cite{Murota95_SICOMP} 
for $\deg \det$; see also \cite[Section 7.1]{MurotaMatrix}.

We assume that the position of
the zero submatrix in FR is in the upper-left corner, 
and recall our notation, such as ${\bf 1}_k$, ${\bf 1}$, 
and $(t^{\alpha})$ in (\ref{eqn:t^alpha}).
%
%For integer vector $\alpha \in \ZZ^n$,
%let  $(t^{\alpha})$ denote the diagonal matrix with %diagonals $t^{\alpha_1}, %t^{\alpha_2},\ldots,t^{\alpha_n}$ in order.
%
%
\begin{description}
	\item[Algorithm: Deg-Det~\cite{HH_degdet}]
	\item[Input:] $B= \sum_{k =1}^m B_kx_k$, and
	a feasible solution $P,Q$ for MVMP.
	\item[Output:] $\deg \Det B$.
	\item[1:]  Solve the problem FR for linear symbolic  matrix $(PBQ)^{(0)}$ to
	obtain optimal $S,T$, where $\ncrank (PBQ)^{(0)} = 2n - r - s$ and
	$S(PBQ)^{(0)}T$ has an $r\times s$ zero submatrix in the upper-left corner.
	\item[2:] If $(PBQ)^{(0)}$ is nc-nonsingular, i.e., $r+s=n$,
	then output $\deg \Det B = -\deg \det P - \deg \det Q$. Otherwise,
	let $(P,Q) \leftarrow ((t^{{\bf 1}_{r}})SP, QT(t^{{\bf 1}_{s}-{\bf 1}}))$,
	and go to step 1.
\end{description}

The algorithm works as follows:
The matrix $SPAQT$ after step~1 has a negative degree in each entry of
its upper-left $r \times s$ submatrix.
Multiplying $t$ for the first $r$ rows and $t^{-1}$
for the last $n-s$ columns produces no entry of positive degree.
%That is,
%$(t^{{\bf 1}_{r}})SP B QT(t^{{\bf 1}_{s}-{\bf 1}})$ has no entry of positive degree.
Thus, the next solution $(P,Q) := \left((t^{{\bf 1}_{r}})SP, QT(t^{{\bf 1}_{s}-{\bf 1}})\right)$
is feasible to MVMP,
and decreases the objective value by $r+s - n (> 0)$.
If the algorithm terminates, then $(P,Q)$ is optimal by Lemma~\ref{lem:optimality}.
If the algorithm does not terminate,
then $\deg \Det B = -\infty$ by (\ref{eqn:weakduality_degDet}).

The \textbf{Deg-Det} algorithm can incorporate the following strategies in steps 1 and 2:
\begin{description}
	\item[{\bf Dominant-MVS strategy}:] In step 1 of {\bf Deg-Det}, choose optimal $S,T$ with maximum $r$ (and minimum $s$).
	\item[{\bf Long-step strategy}:] In step 2 of {\bf Deg-Det}, replace $({\bf 1}_r, {\bf 1}_{s}-{\bf 1})$ by
	$(\kappa {\bf 1}_r, \kappa ({\bf 1}_{s}-{\bf 1}))$
	for some integer $\kappa > 0$ such that the next solution is feasible.
\end{description}
We will show that these strategies bring polynomial iteration complexity 
for linear monomial matrices $A[c]$.

\subsubsection{Euclidean Building}\label{subsub:building}
Analogously to FR viewed as a vector-subspaces optimization MVSP, 
the problem MVMP is viewed as a submodules optimization problem. 
Observe that the change $P,Q \to SP,QT$ for any $S,T \in GL_n(\KK(t)^-)$
keeps the feasibility and the objective value.
Namely, rather than $P,Q$, 
the $\KK(t)^-$-modules generated by $P,Q$ are true variables in MVMP.
Particularly, MVMP is an optimization 
over (the product of) the families of full-rank free $\KK(t)^-$-modules\footnote{Such a module is called a {\em lattice} in the literature.} in $\KK(t)^n$.
These modules form an abstract simplicial complex (by inclusion relation), 
which is known as the {\em Euclidean building} for $GL_n(\KK(t))$; 
see e.g., \cite[Chapter 19]{Garrett}.

In this view, 
the {\bf Deg-Det} algorithm minimizes the objective function by tracing the 1-skeleton of the building, where an adjacent vertex is uniquely determined by
optimal solution $U,V$ of MVSP 
at the current vertex (the modules generated by $P, Q$); see \cite{HH_degdet}.
Therefore, when adopting the dominant-MVS strategy,
the minimizing trajectory in the building produced by {\bf Deg-Det}  
is uniquely determined, though $S,T$ in each step is not unique.

\subsubsection{Linear symbolic monomial matrix}
For a matrix $A = \sum_{k=1}^mA_kx_k$ in (\ref{eqn:A})
and an integer vector $c \in \ZZ^m$, as in (\ref{eqn:A[c]}), 
consider a linear symbolic monomial matrix 
$A[c] :=  \sum_{k=1}^m A_k x_k t^{c_k}$.
We see in Section~\ref{subsec:relationstoCO} below that 
this class of matrices captures the weighted maximization
of several combinatorial optimization problems, 
where $c$ plays the role of a cost vector. 

A natural question from computational complexity viewpoint
is whether one can compute $\deg \Det A[c]$ in time polynomial dependence of $c$. 
The {\bf Deg-Det} algorithm above runs in time pseudo-polynomial in $c$.
Hirai and Ikeda~\cite{HiraiIkeda} addressed this question.
They incorporated {\em cost-scaling}, 
a standard technique in combinatorial optimization, 
with the {\bf Deg-Det} algorithm, and
obtained a polynomial time algorithm for $\deg \Det A[c]$:
\begin{Thm}[\cite{HiraiIkeda}]\label{thm:main_IH} Let $A$ be a matrix in \eqref{eqn:A} and let $c \in \ZZ^m$.
	\begin{itemize}
		\item[{\rm (1)}]
		Suppose that arithmetic operations
		over $\KK$ are performed in constant time.
		Then $\deg \Det A[c]$ can be computed in time polynomial of
		$n,m,$ and $\log C$, where $C := \max_k |c_k|$.
		\item[{\rm (2)}] Suppose that $\KK = \QQ$ and that each $A_k$ consists of integer entries whose absolute values are at most $D> 0$.
		Then $\deg \Det A[c]$ can be computed in time polynomial of
		$n,m, \log C$, and $\log D$.
	\end{itemize}
\end{Thm}

The second result (2) is based on the {\em modulo-$p$ reduction method} by Iwata and Kobayashi~\cite{IwataKobayashi2017} devised for weighted linear matroid matching problem.
As will be mentioned in more detail in Section~\ref{sec:main}, 
$\deg \Det A[c]$ is viewed as a linear optimization 
over the polytope of the nonzero terms of $\det A^{\{d\}} \in \ZZ[x_1,x_2,\ldots,x_m]$ 
for $d \geq n-1$. This fact can be seen from Proposition~\ref{prop:blow-up_degDet}.
Then, an optimal nonzero-term is also 
an optimal nonzero-term $\deg \det A^{\{d\}}$ modulo $p$ for large prime $p$.
By using this fact, the problem is reduced to symbolic matrices $A \bmod p$ 
over the finite field $GF(p)$
for a polynomial number of primes $p$ with polynomial bit-length.

Hirai and Ikeda~\cite{HiraiIkeda} also showed that the $\log C$ dependency can be removed
by the {\em Frank-Tardos method}~\cite{FrankTardos1987}, 
a general procedure rounding a weight vector $c$ to have a bit-complexity polynomial of 
the problem dimension, which is also a standard technique 
to obtain strongly polynomial time algorithms 
in combinatorial optimization.
In Section~\ref{sec:main},   
we develop
a more natural and fast primal-dual type strongly polynomial-time
algorithm to compute $\deg \Det A[c]$, and more generally, 
the maximum degrees of minors of all sizes.

\subsection{Combinatorial optimization}\label{subsec:relationstoCO}
In this section, we explain the interactions of the above results with combinatorial optimization, which is the main theme of this paper.  
\subsubsection{Bipartite matching}
Let $G = (U \sqcup V, E)$ be a bipartite graph with two color classes $U,V$ and edge set $E$.
Suppose (for simplicity) that $U=V=[n]$. 
Define a linear symbolic matrix $A$ ({\em Edmonds matrix} by
\begin{equation}\label{eqn:bipartite}
	A := \sum_{ij \in E} e_i e_j^{\top}x_{ij},
\end{equation}
where $e_i$ denotes the $i$-th unit vector in $\KK^n$.
As is well-known, $\rank A$ is equal to the maximum size of a matching of $G$. 
It is also equal to $\ncrank A$.
To see this, observe that matrices $S,T$ in FR 
can be taken as permutation matrices 
(i.e., vector subspaces $U,V$ in MVSP can be taken as coordinate subspaces).
Then, $\ncrank A$ is equal to the maximum of $2n - |S|$ over all stable sets $S$ in $G$.
By the K\H{o}nig-Egerv\'ary theorem, this quantity equals the maximum number of a matching.
This means $\rank A = \ncrank A$. In addition, the IQS-algorithm (with $0,1$ substitution) can realize the classical augmenting path algorithm. 

Suppose that $G$ has an integer edge-weight $c:E \to \ZZ$.  
Then $\deg \det A[c]$ is nothing but the maximum weight 
of a perfect matching of $G$.
From the unweighted case and Lemma~\ref{lem:optimality}~(2), 
it holds $\deg \Det A[c] = \deg \det A[c]$.
MVMP in Theorem~\ref{thm:duality_B} becomes the LP-dual 
of the perfect matching polytope, which can be seen 
from the \mbox{(block-)}diagonal property of a solution $P,Q$ of MVMP; % Oki: \bmox{..} forbits inner line breaks
see Section~\ref{subsub:mixed} below.
Further, the {\bf Deg-Det} algorithm with the long-step and dominant MVS-strategies can realize the Hungarian method; see~\cite[Section 5.1]{HH_degdet}.

\subsubsection{Linear matroid intersection}
Let ${\bf M}_1$ and ${\bf M}_2$ be two linear matroids represented 
by vectors $a_1,a_2,\ldots,a_m$ and $b_1,b_2,\ldots,b_m$ in $\KK^n$, respectively.
Define a linear symbolic matrix $A$ by
\begin{equation}\label{eqn:A_intersection}
	A := \sum_{k =1}^m a_k b_k^{\top} x_{k}.
\end{equation}
Then $\rank A$ is equal to the maximum size of a common independent set of the two matroids ${\bf M}_1$ and ${\bf M}_2$.  
The feasibility condition of MVSP is: $U^{\bot} \ni a_k$ or $V^{\bot} \ni b_k$ for $k \in [m]$.
$(U^{\bot}, V^{\bot})$ can be chosen as 
$(\vecspan \{a_k\}_{k \in I}, \vecspan \{b_k\}_{k \in [m] \setminus I})$ for $I \subseteq [m]$, 
and the objective value is $r_{1}(I) + r_{2}([m] \setminus {I})$, where $r_i$ is the rank function of ${\bf M}_i$.
Namely, $\ncrank A$ equals $\min_{I \subseteq [m]} r_{1}(I) + r_{2}([m] \setminus {I})$.
By the matroid intersection theorem (see \cite[Theorem 41.1]{SchrijverBook}), 
it equals $\rank A$.
Thus, the rank and nc-rank are the same.
Again, the IQS-algorithm can realize the matroid intersection algorithm by Edmonds.

The weighted extension goes in a similar manner as bipartite matching.
Given a weight $c:\RR^m \to \ZZ$, $\deg \det A[c]$ is equal to the maximum weight of a common base of the two matroids. Again, from the unweighted case and Lemma~\ref{lem:optimality}~(2), it holds $\deg \det A[c] = \deg \Det A[c]$.  Transforming MVMP to the LP-dual 
of a linear optimization over the common base polytope is also possible but requires a little effort. 
Furue and Hirai~\cite{FurueHirai}
showed that the {\bf Deg-Det} algorithm 
with the two strategies realizes the {\em weight-splitting algorithm}~\cite{Frank1981} with
a new matrix implementation.

\subsubsection{Nonbipartite matching and linear matroid matching}\label{subsub:nonbipartite}
A representative example where the rank and nc-rank differ is the Tutte matrix of a nonbipartite graph $G = ([n],E)$. The Tutte matrix $A_G$ of $G$ is a linear symbolic matrix
\begin{equation}\label{eqn:Tutte}
A = \sum_{ij \in E} (e_i {e_j^{\top}} - e_j e_i^{\top})x_{ij}.
\end{equation}
The maximum matching number of $G$ equals $(1/2)\rank A$.
The rank of the Tutte matrix of $K_3$ is $2$, 
whereas the nc-rank is $3$. So the rank and nc-rank differ.
Interestingly, the nc-rank of $A$
still has a natural combinatorial interpretation: It equals 
twice the {\em fractional} matching number of $G$. 
This fact was recently revealed by Oki and Soma~\cite{OkiSoma}.
They further revealed that this relation is 
generalized for matroid matching.

A {\em matroid matching} for
a collection ${\cal H} = \{H_k\}_{k \in [m]}$ of $2$-dimensional subspaces in $\KK^n$ 
is  a subset $I \subseteq [m]$ with $2 |I| = \dim \sum_{k \in I} H_k$.
It is known~\cite{Lovasz89} that
the maximum cardinality of a matroid matching is equal to twice the rank of
\begin{equation}\label{eqn:linearmatroidmatching}
A_{\cal H} := \sum_{k \in [m]} (a_k {b_k^{\top}} - b_k a_k^{\top})x_{k},
\end{equation}
where $a_k,b_k$ is any basis of subspace $H_k$ 
for $k \in [m]$. 
Oki and Soma  showed 
that the nc-rank of $A_{\cal H}$ is equal to twice the maximum {\em fractional} matroid matching number. In Section~\ref{subsec:frac_linear}, 
we study this linear symbolic matrix $A_{\cal H}$ in more detail 
and establish an analogous relation 
for $\deg \Det A[c]$.
Note that
$\deg \det A_{\cal H}[c]$
is equal to twice the maximum weight of a matroid matching, and that
a polynomial-time algorithm
for weighted linear matroid matching, found recently 
by Iwata and Kobayashi~\cite{IwataKobayashi2017}, 
is based on a related deg-det formulation.

\subsubsection{Mixed matrix and partitioned matrix}\label{subsub:mixed} 
A {\em mixed matrix} (Murota~\cite{MurotaMatrix}) 
is the sum $Q+T$ of a matrix $Q$ over $\KK$ and the Edmonds matrix $T$ for a bipartite graph $G$,  
which is also viewed as a linear symbolic matrix $A = Q x_0 + \sum_{ij \in E}E_{ij}x_{ij}$, where $E_{ij} = e_i e_j^\top$.
From Murota's formula~\cite[Section 4.2]{MurotaMatrix}, 
its rank is interpreted as the optimal value of MVSP. 
Thus the rank and nc-rank are the same.
A {\em mixed polynomial/rational matrix}~\cite{MurotaMatrix} 
is $Q(t) x_0 + \sum_{ij \in E}E_{ij}x_{ij}t^{c_{ij}}$, 
where $Q(t)$ is a polynomial/rational over $\KK$
and $c: E \to \ZZ$ is a weight on $E$.
Via Lemma~\ref{lem:optimality}~(2), $\deg \det$ 
and $\deg \Det$ are equal.
These classes of matrices and their rank/deg-det computation are motivated 
by numeric-symbolic integrated analysis on dynamical systems, 
enhanced with methods of combinatorial optimization.

A related engineering-motivated class of matrices is (generic) {\em partitioned matrices} (Ito, Iwata, and Murota~\cite{ItoIwataMurota94}), which are
linear symbolic matrices of form:
\begin{equation} \label{eqn:2x2}
	A = \left(
	\begin{array}{ccccc}
		A_{11}x_{11} & A_{12} x_{12} &\cdots & A_{1 n} x_{1 n}\\
		A_{21}x_{21} & A_{22} x_{22}&\cdots & A_{2 n} x_{2 n} \\
		\vdots & \vdots & \ddots & \vdots \\
		A_{n1} x_{n1}&A_{n2} x_{n2}&\cdots & A_{n n} x_{n n}
	\end{array}\right),
\end{equation}
where $A_{ij}$ is an $n_i \times m_j$ matrix  over $\KK$ for $i,j \in [n]$. 
In fact, (nc-)rank computation of any linear symbolic matrix reduces to a matrix of this form; see \cite[Appendix]{HH_degdet}. 
An important feature of a partitioned matrix is: 
There is an optimal solution $S,T$ in FR that is {\em block-diagonalized} as
 \begin{align}\label{eqn:block-diagonalized}
    S = \left(
	\begin{array}{cccc}
		S_1 &  &  &  \\
		& S_2 & &  \\
            & & \ddots & \\
            & & & S_n
	\end{array}
	\right),
 \quad
 T = \left(
	\begin{array}{cccc}
		T_1 &  &  &  \\
		& T_2 & &  \\
            & & \ddots & \\
            & & & T_n
	\end{array}
	\right), 
    \end{align}
    where $S_i \in GL_{n_i}(\KK)$ and 
    $T_j \in GL_{m_j}(\KK)$ for $i,j \in [n]$. 

Observe that the Edmonds matrix in (\ref{eqn:bipartite}) is  
the special case where each submatrix $A_{ij}$ is $1 \times 1$. 
Iwata and Murota~\cite{IwataMurota95} studied 
the next case where
each submatrix $A_{ij}$ is $2 \times 2$.   
They proved a formula (good characterization) of the rank, which 
turned out to be equal to the Fortin-Rautenauer formula (Theorem~\ref{thm:FR}). 
Hence the rank and nc-rank are the same for $2\times 2$ partitioned matrices.
Along with the development of the nc-rank theory,
Iwamasa and Hirai~\cite{HiraiIwamasaIPCO}
took up the rank-computation of $2 \times 2$ partitioned matrices
as a {\em $2$-dimensional generalization} of bipartite matching.
They showed that the rank equals 
the maximum of certain algebraically-constraint $2$-matchings
in the bipartite graph of the nonzero block pattern of $A$, 
and sharpened the IQS-algorithm to develop 
a combinatorial polynomial-time augmenting path algorithm 
to compute a maximum matching.

The deg-det 
computation for weighted $2\times 2$ partitioned matrices $A[c]$, the maximum weight matching problem in this setting, are addressed by Hirai and Ikeda~\cite{HiraiIkeda} and Iwamasa~\cite{Iwamasa_degdet}, 
where $\deg \det A[c]= \deg \Det A[c]$ holds again from Lemma~\ref{lem:optimality}~(2).  
The former~\cite{HiraiIkeda} provided a cost-scaling polynomial-time algorithm using the unweighted algorithm of \cite{HiraiIwamasaIPCO} as a subroutine, which is a direct application of Theorem~\ref{thm:main_IH}. 
The latter~\cite{Iwamasa_degdet} provided a strongly polynomial-time primal-dual algorithm to 
compute the maximum degrees of minors of all sizes, and obtained algorithmically
a new min-max theorem (good characterization) for the maximum degrees of minors.
In Section~\ref{sec:main}, we extend his min-max formula for general linear monomial matrices, 
and develop a unified strongly polynomial-time primal-dual framework.

\section{Maximum degrees of subdeterminants}\label{sec:main}

In this section, we present an improved analysis and results 
for the deg-Det computation.
In addition to $\deg \det$ and $\deg \Det$,
we consider the maximum degrees of subdeterminants.
Let $B$ be a matrix in (\ref{eqn:B}).
For $\ell \in [n]$, define the maximum degrees over $\ell \times \ell$ subdeterminants of $B$ by
\begin{eqnarray*}%\label{eqn:delta_star}
 	\delta_\ell (B) &:= & \max \{ \deg \det B[I,J]  \mid I,J \subseteq [n]: |I| = |J| =\ell\}, \\
 	\varDelta_\ell(B) &:= & \max \{ \deg \Det B[I,J]  \mid I,J \subseteq [n]: |I| = |J| =\ell\},
 \end{eqnarray*}
where $B[I,J]$ denote the submatrix of $B$ having row set $I$ and column set $J$.
It is basic (see e.g., \cite[section 5.1.2 (5.5)]{MurotaMatrix}) that $\delta_\ell (B) = \delta_\ell (PBQ)$ holds for any biproper $P,Q$. 
This property holds for $\varDelta_\ell$:
\begin{equation}\label{eqn:Delta(B)=Delta(PBQ)}
\varDelta_\ell (B) = \varDelta_\ell (PBQ) \quad (P,Q \in GL_n(\KK(t)^-)).
\end{equation}
See \cite[Proposition 2.9]{HH_degdet}.
 For convention, we define the degree of the determinant of an empty matrix by $\delta_0(B) = \varDelta_0(B) := 0$.
Let $\delta_{\rm max}(B)$ and $\varDelta_{\rm max}(B)$ 
denote the maximum degrees of all possible subdeterminants:
%\[
%\delta_{\rm max}(B) :=  \max_{\ell \in [0,n]} \delta_{\ell}(B), 
%\quad \varDelta_{\rm max}(B) := \max_{\ell \in [0,n]} \varDelta_{\ell}(B).
%\]
\begin{eqnarray*}
\delta_{\rm max}(B) &:= & \max_{\ell \in [0,n]} \delta_{\ell}(B), \\
\varDelta_{\rm max}(B) &:=& \max_{\ell \in [0,n]} \varDelta_{\ell}(B).
\end{eqnarray*}
The correspondents in combinatorial optimization
are clear: $\delta_{\ell}$ and $\Delta_\ell$
represent maximum weights under the cardinality constraint, 
whereas $\delta_{\rm max}$ and $\Delta_{\rm max}$ represent 
unconstraint maximum weights.

We introduce additional notation.
For a subset $D \subseteq \RR$, 
let $D^n_\downarrow$ denote the set of vectors $\alpha \in D^n$ with $\alpha_1 \geq \alpha_2 \geq \cdots \geq \alpha_n$.
Similarly, $D^n_{\uparrow}$ denotes the set of vectors $\alpha \in D^n$ with $\alpha_1 \leq \alpha_2 \leq \cdots \leq \alpha_n$.
A zero matrix is denoted by $O$ and by $O_{n,m}$ 
if it has $n$ rows and $m$ columns.

The arguments in this section are
adapted to rectangular matrices $B$ 
(by e.g., filling zeros to make $B$ being a square matrix).

\subsection{Linear symbolic rational matrices}
%First, we analyze $\varDelta_\ell$ and $\varDelta_{\rm max}$ 
%for linear symbolic rational matrices.
%The main results  
%are duality theorems for $\varDelta_{\ell}$ 
%and $\varDelta_{\rm max}$ (Theorem~\ref{thm:varDelta}) 
%that extends Theorem~\ref{thm:duality_B},
%and an extension of {\bf Deg-Det}
%to compute all $\varDelta_{\ell}$ (Section~\ref{subsub:DegSubDet}).

\subsubsection{Duality theorem}
The duality theorems for $\varDelta_{\ell}$ and $\varDelta_{\max}$ are as follows.
\begin{Thm}\label{thm:varDelta}
Let $B = B(t)$ be a matrix in {\rm (\ref{eqn:B})}.
Then $\varDelta_\ell(B)$ is equal to the optimal value of the following problem:
	\begin{eqnarray*}
		\mbox{\rm (MVMP$_{\ell}$)} \quad {\rm Min.} && - \sum_{i=n-\ell+ 1}^n \alpha_i- \sum_{j=n-\ell + 1}^n \beta_j \\
		{\rm s.t.} && \deg ((t^{\alpha})PB_k Q(t^{\beta}))_{ij} \leq 0 \quad (i,j \in [n], k \in [m]), \\
		&& \alpha, \beta \in \ZZ^n_{\downarrow}, P,Q \in GL_n (\KK(t)^-).
	\end{eqnarray*}
In addition, $\varDelta_{\rm max}(B)$ is equal to the optimal value of the following problem:
	\begin{eqnarray*}
		\mbox{\rm (MVMP$_{\rm max}$)} \quad {\rm Min.} && \sum_{i=1}^n \xi_i+ \sum_{j= 1}^n \eta_j \\
		{\rm s.t.} && \deg ((t^{-\xi})PB_k Q(t^{-\eta}))_{ij} \leq 0 \quad (i,j \in [n], k \in [m]), \\
		&& \xi, \eta \in \ZZ^n_{+}, P,Q \in GL_n (\KK(t)^-).
	\end{eqnarray*}
\end{Thm}
In the case where $B$ is a mixed rational matrix (Section~\ref{subsub:mixed}),
this theorem provides new explicit formulas of $\delta_{\ell} (= \varDelta_{\ell})$, 
$\delta_{\rm max} (= \varDelta_{\rm max})$. 
The proof is given in Section~\ref{subsub:proof} after 
studying general properties of the {\bf Deg-Det} algorithm in Section~\ref{subsub:properties}.
We here verify the weak duality in MVMP$_{\rm \ell}$.
For any feasible $(t^{\alpha})P, Q(t^{\beta})$ and $I,J \subseteq [n]$, it holds
\begin{equation*}
0 \geq \deg \Det ((t^{\alpha}) PBQ (t^{\beta}))[I,J] = \sum_{i \in I }\alpha_i + \sum_{j \in J} \beta_j + \deg \Det PBQ[I,J].
\end{equation*}
Hence we obtain
\begin{equation}\label{eqn:weak_duality0}
\deg \Det PBQ[I,J] \leq - \sum_{i \in I} \alpha_i - \sum_{j \in J} \beta_j \leq - \sum_{i=n-\ell+ 1}^n \alpha_i- \sum_{j=n-\ell + 1}^n \beta_j .
\end{equation}
Then, by (\ref{eqn:Delta(B)=Delta(PBQ)}), it holds 
\begin{equation}\label{eqn:weak_duality}
	\varDelta_{\ell}(B) = \varDelta_{\ell}(PBQ) \leq - \sum_{i=n-\ell+ 1}^n \alpha_i- \sum_{j=n-\ell + 1}^n \beta_j .
\end{equation}
Note that this upper bounding idea (for $\delta_{\ell}(B)$) already appeared
in Iwata, Murota, and Sakuta~\cite{IwataMurotaSakuta1996}.

\subsubsection{Analysis on the {\bf Deg-Det} algorithm}\label{subsub:properties}

We start with two preliminary lemmas.
The first one is a partial commutative property 
between $(t^{\alpha})$ and triangular matrices. 
The proof is straightforward.
%
%We will often use
%the following observation:
\begin{Lem}\label{lem:commtative}
Let $U \in GL_n(\KK(t)^-)$ be upper-unitriangular biproper, and let $\alpha \in \ZZ^n_\downarrow$. Then
\[
U (t^{\alpha}) = (t^{\alpha}) \tilde{U},
\]
where $\tilde U := (t^{-\alpha})U (t^{\alpha}) \in GL_n(\KK(t)^-)$ is upper-unitriangular biproper.
\end{Lem}
The second one is a version of LU-decomposition, 
called the {\em Bruhat decomposition}.
We identify a permutation $\pi:[n] \to [n]$ with a permutation matrix $\sum_{i=1}^n e_i e^{\top}_{\pi(i)}$.
\begin{Lem}[{Bruhat decomposition}]\label{lem:Bruhat}
	Any nonsingular matrix $S \in GL_n(\KK)$ can be decomposed as
	\[
	S = L \pi_S U
	\]
	for permutation matrix $\pi_S$, lower-triangular matrix $L$, and upper-triangular matrix $U$, where $\pi_S$ is uniquely determined.
\end{Lem}
For completeness, we give a proof based on \cite[Theorem 9.2.2]{Cohn_Algebra}, which works on a matrix over any skew field and can deduce the Bruhat normal form.
\begin{proof}
	For $i=1,2,\ldots,n$ in order,
	choose the smallest index $\pi(i):= j \in [n]$ with $S_{ij} \neq 0$.
	Add scalar multiples of the $i$-th row and $j$-th column to later rows and columns
	to make $S_{i'j} = S_{ij'} = 0$ for all $(i',j') \neq (i,j)$.
	By the nonsingularity of $S$, this is actually possible, and $\pi$ is a permutation on $[n]$.
        This procedure
	corresponds to multiplying to $S$ a lower-triangular matrix $L$ from the left
	and an upper-triangular matrix $U$ from the right.
	The resulting matrix is $\pi D$, where $D$ is a diagonal matrix. Namely $LSU = \pi D$,
	and we obtain desired decomposition $S = L^{-1} \pi (DU^{-1})$.

	To show the uniqueness, it suffices to show that
	$L\pi' U= \pi$ implies $\pi' = \pi$.
	This is immediate from the following observation:
	If $\pi'_{ij} = 1$, the multiplication
	of a lower/upper-triangular matrix from left/right to $\pi'$
	does not affect the $(i,j)$-entry (since $\pi'_{i'j} = \pi'_{ij'} =0$ for $i' < i$, $j' < j$).
	Namely, $\pi_{ij}' = \pi_{ij} = 1$, implying $\pi'=\pi$.
\end{proof}
%

%Note that the Bruhat decomposition is usually referred to as in
%the case where $L$ is also upper-triangular but
%this case is also obtained by an easy adjustment of the above proof.

\begin{Lem}\label{lem:(S,T)}
	\begin{itemize}
		\item[{\rm (1)}] In step 1 of {\bf Deg-Det}, optimal $S,T$ can be chosen in the form
		\begin{equation}\label{eqn:chosen}
			S = \pi_S U,\ T = M \pi_T,
		\end{equation}
		where $U$ is upper-triangular and $M$ is lower-triangular.
		\item[{\rm (2)}] In addition, if $\ncrank (PBQ)^{(0)} = \ell$
		and the lower-right $\ell \times \ell$ submatrix of $(PBQ)^{(0)}$ is nc-nonsingular,
		then the above $S,T$ can be chosen as
		\begin{equation}\label{eqn:pi_n-k}
			\pi_S[n-\ell]  = [n-\ell],\ \pi_T[n-\ell] = [n-\ell],
		\end{equation}
		and the lower-right $\ell \times \ell$ submatrix of $((t^{{\bf 1}_r})SPBQT(t^{{\bf 1}_s- {\bf 1}}))^{(0)}$ is nc-nonsingular.
	\end{itemize}
\end{Lem}
\begin{proof}
	Let $S,T$ be an optimal solution in step 1.
	Consider the Bruhat decomposition $S = L \pi_S U$ and $T = M \pi_T V$, where $L,M$ are lower-triangular and $U,V$ are upper-trianglar.
	Let $A := (PBQ)^{(0)}$.
	Then $SAT$ has an $r \times s$ zero submatrix in upper-left corner,
	and so does $L^{-1} SAT U^{-1} = \pi_S U A M \pi_T$, which implies (1).

	(2). Let $S,T$ be an optimal solution in step 1 (not necessarily the form of (\ref{eqn:chosen})).
        We claim:
	\begin{Clm}
		$\rank S[[r],[n-\ell]] = \rank T[[n-\ell], [s]] = n-\ell$.
	\end{Clm}
	\begin{proof}
		Let $a := \rank S[[r],[n-\ell]]$ and $b:= T[[n-\ell], [s]]$. By adding $n-\ell \geq a$ and $n-\ell \geq b$ (with $\ell = 2n-r-s$), we have
		\begin{equation*}
			2(r+s - n) \geq a + b.
		\end{equation*}
		By row elimination within the first $r$ rows,
		we may assume $S$ to satisfy $S[[r-a],[n-\ell]] = O$.
		Similarly, we may assume $T[[n-\ell],[s-b]] = O$.

		Let $S^* := S[[r-a], [n] \setminus [n-\ell]]$, $T^* := T[[n] \setminus [n-\ell], [s-b]]$, and
		let $A^*$ denote the lower-right $\ell \times \ell$ submatrix of $A$,
        which is nc-nonsingular by the assumption.
		From $S[[r],[n]]AT[[n],[s]] = O$, we have
		\[
		S^{*} A^* T^* = O_{r-a, s-b}.
		\]
		Since $A^*$ has nc-rank $\ell$, the sum $(r-a)+(s-b)$
		of the row size of $S^*$ and column size of $T^*$ is at most $\ell (= 2n-r-s)$.
		Thus $r-a + s- b \leq 2n-r-s$, and
		\begin{equation*}
			2(r+s -n) \leq a+b.
		\end{equation*}
		Therefore the equality holds, which implies $n-\ell= a =b$.
	\end{proof}

	We can assume that $S[[n-\ell],[n-\ell]]$ and $T[[n-\ell],[n-\ell]]$ are nonsingular.
	Then, in the proof of Bruhat decomposition (Lemma~\ref{lem:Bruhat}),
	for row index $i \leq n-\ell$, column index $\pi_S(i) = j \leq n-\ell$ is chosen.
	This means that $\pi_S[n-\ell] = \pi_T [n-\ell] = [n-\ell]$.
        In the Bruhat decomposition $S= L \pi_S U$, $T = M \pi_T V$,
	it holds
	\[
	\pi_S U= \left(
	\begin{array}{cc}
		S_1 & \ast \\
		O_{\ell,n-\ell} & S_2
	\end{array}
	\right), \quad  M \pi_T  = \left(
	\begin{array}{cc}
		T_1 & O_{n-\ell,\ell}  \\
		\ast & T_2
	\end{array}
	\right),
	\]
	where $S_i,T_i$ are nonsingular.
	The lower-right $\ell \times \ell$ submatrix of $\pi_S U A M \pi_T$
	is written as $S_2 A^* T_2$ that is nc-nonsingular.
	On the other hand,  the upper-left $r \times s$ submatrix of $\pi_S U A M \pi_T$
	is the zero matrix. $S_2 A^* T_2$ has a block structure as
	\begin{equation}\label{eqn:before}
	S_2 A^* T_2 = \left(
	\begin{array}{cc}
		O_{r',s'} & A_2^* \\
		A_1^* & \ast
	\end{array}
	\right),
	\end{equation}
	where $r' := \ell - (n-r) = n-s$, $s' := \ell - (n-s) = n-r$, and $r'+s' =\ell$.
	Necessarily both $A_1^*$ and $A_2^*$ are (square) nc-nonsingular matrices.
	Now,
	the lower-right $\ell \times \ell$ submatrix of $((t^{{\bf 1}_r})\pi_S U A M \pi_T(t^{{\bf 1}_s-{\bf 1}}))^{(0)}$
	is given as
	\begin{equation}\label{eqn:after}
	\left(
	\begin{array}{cc}
		\ast & A_2^* \\
		A_1^* & O_{s',r'}
	\end{array}
	\right),
	\end{equation}
	which is nc-nonsingular, as required.
\end{proof}

\begin{Lem}\label{lem:consecutive}
	In two consecutive iterations of {\bf Deg-Det},
        suppose that feasible solution $P,Q$ is updated as
	\[
	P,Q \ \to \ (t^{{\bf 1}_r})SP,QT(t^{{\bf 1}-{\bf 1}_s}) \ \to \ (t^{\11_{r'}})S'(t^{\11_r})SP,QT(t^{\11_s-\11})T'(t^{\11_{s'}-\11}).
	\]
 Then we have the following:
	\begin{itemize}
		\item[{\rm (1)}] $r'+s' \leq r+s$; in particular, $\ncrank (PBQ)^{(0)}$ is monotone nondecreasing.
		\item[{\rm (2)}] Suppose that the dominant-MVS strategy is used.
        If $r'+s' = r + s$, then $\pi_{S'}[r'] \subseteq [r]$ and $\pi_{T'}^{-1}[s'] \supseteq [s]$.
 %
 % Then it holds:
%		\begin{itemize}
%			\item[(2-1)] If $r'+s' = r + s$, then $r' \leq r$ and $s' \geq s$.
%			\item[(2-2)] If $r' =r$ and $s' = s$, then $\pi_{S'}[r] = [r]$ and $\pi_{T'}[s] = [s]$.
%		\end{itemize}
	\end{itemize}
\end{Lem}
\begin{proof}
	We can assume for simplicity that $SP = QT = I$ (by replacing $B$ with $SPBQT$).
	Consider Bruhat decomposition $S' = L \pi_{S'} U$ and $T' = M \pi_{T'} V$, where $L,M$ are lower-triangular and $U,V$ are upper-triangular.
 In the view of Lemma~\ref{lem:commtative}, we have
        \begin{eqnarray*}
        (t^{\11_{r'}})S'(t^{\11_r}) & = & (t^{\11_{r'}}) L \pi_{S'} U (t^{\11_r})
        = \tilde L (t^{\11_{r'}}) \pi_{S'}  (t^{\11_r}) \tilde U
        \\
        & = & \tilde L\pi_{S'}\pi_{S'}^{-1} (t^{\11_{r'}}) \pi_{S'}  (t^{\11_r}) \tilde U = \tilde L\pi_{S'} (t^{\11_{\pi_{S'}[r']}+\11_r}) \tilde U,
        \end{eqnarray*}
        where $\tilde L\pi_{S'}$ and $\tilde U$ are biproper.
	Similarly, we have
	\begin{eqnarray*}
	(t^{\11_s-\11})T'(t^{\11_{s'}-\11}) &=& (t^{\11_s-\11})M \pi_{T'} V(t^{\11_{s'}-\11})  = \tilde M  (t^{\11_s-\11}) \pi_{T'} (t^{\11_{s'}-\11}) \tilde V\\
 &=& \tilde M  (t^{\11_s-\11}) \pi_{T'} (t^{\11_{s'}-\11}) \pi_{T'}^{-1}\pi_{T'}\tilde V =
 \tilde M (t^{\11_s-\11+ \11_{\pi^{-1}_{T'}[s']}- \11}) \pi_{T'}\tilde V 
	\end{eqnarray*}
	for biproper $\tilde M,\tilde V$.

	Let $(X,Y,X',Y'):=([r], [s], \pi_{S'}[r'], \pi^{-1}_{T'}[s'])$.
	Then $(t^{\11_{X'}+\11_{X}}) \tilde U,  \tilde M(t^{\11_{Y}-\11+ \11_{Y'}-\11})$ is feasible.
	This implies that $(\tilde UB\tilde M)^{(0)}$ has the following two zero blocks:
	\begin{equation}\label{eqn:two_zero_blocks}
		(\tilde UB\tilde M)^{(0)}[X \cup X', Y \cap Y'] = O,\ (\tilde UB\tilde M)^{(0)}[X \cap X', Y \cup Y'] = O.
	\end{equation}
	Also we have
	\begin{equation}
		r'+s' + r + s = |X'|+|Y'| + |X|+|Y|= |X \cap X'| + |Y \cup Y'| +  |X \cup X'| + |Y \cap Y'|.
	\end{equation}
	On the other hand, the nc-rank of $B^{(0)}$ and $(\tilde UB\tilde M)^{(0)}$ are the same:
	\[
	\ncrank B^{(0)} = \ncrank \tilde U^{(0)} B^{(0)} \tilde M^{(0)} = \ncrank (\tilde UB\tilde M)^{(0)}.
	\]

	(1). If $r'+s' > r + s$, then one of the zero blocks (\ref{eqn:two_zero_blocks}) has size larger than $r+s$,
	which leads to a contradiction $\ncrank (\tilde UB\tilde M)^{(0)} < 2n-r-s = \ncrank B^{(0)} = \ncrank (\tilde UB\tilde M)^{(0)}$.

    (2). It necessarily holds $r'+s' = r+s = |X \cap X'| + |Y \cup Y'| = |X \cup X'| + |Y \cap Y'|$.
    If $X' \not \subseteq X$, then
    $|X \cup X'| > |X| =r$,
    and the zero block $(\tilde UB\tilde M)^{(0)}[X \cup X', Y \cap Y']$ has row size larger than $r$,
    which contradicts the dominant-MVS strategy.
    Thus $X' \subseteq X$, and $|Y \cap Y'| = |Y|$, implying $Y \subseteq Y'$.
%
%	(2-1). It necessarily holds $r'+s' = r+s = |X \cap X'| + |Y \cup Y'| = |X \cup X'| + |Y \cap Y'|$.
%	If $r' > r$ and $s < s'$, then the zero block
%	$(\tilde UB\tilde M)^{(0)}[X \cup X', Y \cap Y']$ has rows larger than $r$,
%	which contradicts the dominant-MVS strategy.
%	(2-2). If $X \neq X'$ (i.e., $[r] \neq \pi_{S'}[r]$),
%	then the same contradiction holds.
%	Thus $X=X'$, and necessarily $Y=Y'$.
\end{proof}

\subsubsection{Proof of Theorem~\ref{thm:varDelta}}\label{subsub:proof}
Let $(t^{\alpha})P,Q(t^{\beta})$ be a feasible solution in MVMP$_\ell$.
Suppose that the following conditions (complementary slackness) hold:
\begin{itemize}
%	\item[(C0)] $\ncrank ((t^{\alpha})PBQ(t^{\beta}))^{(0)} = \ell$.
	\item[(C1)] $\alpha_1 = \cdots = \alpha_{n-\ell} \geq \alpha_{n-\ell +1} \geq \cdots \geq \alpha_{n}$ and $\beta_1 = \cdots = \beta_{n-\ell} \geq \beta_{n-\ell +1} \geq \cdots \geq \beta_{n}$.
	\item[(C2)] The $\ell \times \ell$ lower-right submatrix of $((t^{\alpha})PBQ(t^{\beta}))^{(0)}$ is nc-nonsingular.
\end{itemize}
In this case, by taking the last $\ell$ rows and $\ell$ columns as $I$ and $J$, respectively, it holds
\[
0 = \deg \Det (t^{\alpha})PBQ(t^{\beta}) [I,J] = \deg \Det(PBQ)[I,J] + \sum_{i=n-\ell+ 1}^n \alpha_i+ \sum_{j=n-\ell + 1}^n \beta_j.
\]
This means that the inequalities in (\ref{eqn:weak_duality0}) and (\ref{eqn:weak_duality}) hold in equality:
\begin{equation*}
	\varDelta_{\ell}(B) = \varDelta_{\ell}(PBQ) =  - \sum_{i=n-\ell+ 1}^n \alpha_i- \sum_{j=n-\ell + 1}^n \beta_j.
\end{equation*}

We are going to show that there is a feasible solution
$(t^{\alpha})P, Q(t^{\beta})$ satisfying (C1), (C2) for all $\ell \in [n]$
with $\varDelta_{\ell}(B) > -\infty$, and that MVMP$_\ell$ is unbounded for other $\ell$.
Initially, let $P=Q = I$,
$\alpha := 0$, and $\beta = - d {\bf 1}$
for the maximum degree $d (=\varDelta_{1}(B) = \delta_{1}(B))$ of entries in $B$.
Then $(t^{\alpha})P,Q(t^{\beta})$ is obviously feasible.
Let $\ell := \ncrank ((t^{\alpha})PBQ(t^{\beta}))^{(0)}$.
By permuting rows and columns,
it satisfies (C1) and (C2) for all $\ell' \leq \ell$.

Consider to apply one iteration of {\bf Deg-Det} from $(t^{\alpha})P,Q(t^{\beta})$, where optimal $(S,T)$
is chosen as in Lemma~\ref{lem:(S,T)}.
The next solution is $(t^{{\bf 1}_r})\pi_S U (t^{\alpha})P, Q(t^{\beta})M \pi_T (t^{\11_s -\11})$, which is rewritten as
\begin{equation}
	(t^{\11_r})\pi_S U (t^{\alpha})P = \pi_S (t^{\11_{\pi_S[r]} + \alpha}) \tilde U P, \quad
	Q(t^{\beta}) M \pi_T (t^{\11_s-\11}) =Q\tilde M(t^{\beta+\11_{\pi_T^{-1}[s]}-\11}) \pi_T
\end{equation}
for biproper $\tilde U, \tilde M$.
By (\ref{eqn:pi_n-k}), both $\pi_S[r]$ and $\pi_T^{-1}[s]$ include $[n-\ell]$, where $n -\ell \leq \min (r,s)$.
Permute the last $\ell$ rows and columns to obtain
a feasible solution $(t^{\alpha'})P', Q'(t^{\beta'})$ satisfying (C1), (C2) with $\ell$.
For $\ell' \geq \ell$, it holds
\begin{eqnarray*}
	- \sum_{i=n-\ell'+ 1}^n \alpha'_i- \sum_{j=n-\ell' + 1}^n \beta'_j &= &- \sum_{i=n-\ell'+ 1}^n \alpha_i - (r-n+\ell') - \sum_{j=n-\ell' + 1}^n \beta_j + (n-s)\\
	&=&  - \sum_{i=n-\ell'+ 1}^n \alpha_i- \sum_{j=n-\ell' + 1}^n \beta_j - (\ell' - \ell).
\end{eqnarray*}
This means that the upper bound of $\varDelta_{\ell'}(B)$ strictly decreases if $\ell' > \ell$.

Case 1: Suppose that $\ncrank ((t^{\alpha'})P'BQ'(t^{\beta'}))^{(0)} = \ell + \delta$ for $\delta > 0$.
The last $\ell \times \ell$ submatrix can be extended to an nc-nonsingular $(\ell + \delta) \times (\ell + \delta)$ submatrix.
By permuting the first $n-\ell$ rows and columns, we may assume that
the last $(\ell + \delta') \times (\ell + \delta')$ submatrix of $((t^{\alpha'})P'BQ'(t^{\beta'}))^{(0)}$
is nc-nonsingular for all $\delta' \leq \delta$.
Thus, $(t^{\alpha'})P', Q'(t^{\beta'})$ satisfies (C1), (C2)
with $\ell + \delta'$ for all $\delta' \leq \delta$.

Case 2: Suppose that $\ncrank ((t^{\alpha'})P'BQ'(t^{\beta'}))^{(0)} = \ncrank ((t^{\alpha})PBQ(t^{\beta}))^{(0)} = \ell$.
In this case, the new feasible solution $(t^{\alpha'})P', Q'(t^{\beta'})$
still satisfies (C1), (C2) with $\ell$.
Let $(t^\alpha)P, Q(t^\beta) \leftarrow (t^{\alpha'})P', Q'(t^{\beta'})$.
Repeat the iterations until $\ncrank(t^{\alpha})PBQ(t^{\beta})^{(0)}$
increases. Then it reduces to the above case.
If it repeats infinitely many steps,
then MVMP$_\ell$ is unbounded and
$\varDelta_{\ell'}(B) = - \infty$ for all $\ell' > \ell$, as desired.

Next, we show the latter part of the theorem, the formula for $\varDelta_{\rm max}(B)$. 
From the former part, we have
\begin{eqnarray*}
\varDelta_{\rm max}(B) &= & \max_{\ell \in [0,n]} 
\varDelta_{\ell}(B) = \max_{\ell\in [0,n]} \min_{(t^{\alpha})P,Q(t^{\beta})} - \sum_{i=n-\ell+1}^n \alpha_i - \sum_{j=n-\ell+1}^n \beta_j \\
&\leq& \min_{(t^{\alpha})P,Q(t^{\beta})} \max_{\ell\in [0,n]} - \sum_{i=n-\ell+1}^n \alpha_i - \sum_{j=n-\ell+1}^n \beta_j \\
&=& \min_{(t^{\alpha})P,Q(t^{\beta})} \sum_{i=1}^n \max \{0, - \alpha_i - \beta_i\},
\end{eqnarray*}
where the inequality follows from general property 
$\max_y \min_x f(x,y) \leq \min_x \max_y f(x,y)$ 
of a bifunction $f$. 
In the last minimization, 
we may assume that $\alpha_i,\beta_j$ are nonpositive.
By letting $\xi := -\alpha$, $\eta := -\beta$,
we obtain MVMP$_{\rm max}$.
Therefore, it suffices to show that 
the inequality holds in the equality. 
This follows from the existence of 
a saddle-point\footnote{ 
A saddle-point $(x^*,y^*)$ of a function $f: D \times E \to \RR$
is a point $(x^*,y^*)$ satisfying $\min_{x \in D} f(x,y^*) = f(x^*,y^*) = \max_{y \in E} f(x^*,y)$. In this case, equality $\min_{x \in D} \max_{y \in E}f(x,y) =\max_{y \in E} \min_{x \in D}f(x,y)$ holds. Indeed, 
from $\max_{y \in E}\min_{x \in D} f(x,y) \geq \min_{x \in D} f(x,y^*)$ 
and $\max_{y \in E} f(x^*,y) \geq \min_{x \in D} \max_{y \in E} f(x,y)$, 
it holds  $\max_{y \in E}\min_{x \in D} f(x,y) \geq \min_{x \in D} \max_{y \in E} f(x,y)$
On the other hand, the reverse inequality always holds.
}
for bifunction 
$\ell, ((t^{\alpha})P,Q(t^{\beta})) \mapsto - \sum_{i=n-\ell+1}^n \alpha_i + \beta_i$, which we are going to show.

In the above proof of the former part, we observe:
\begin{itemize}
\item $\alpha_1$ is increasing from $0$, and $\beta_1 = -d$ is constant.
\item In Case 1, it holds $\varDelta_{\ell+ \delta'}(B) = \varDelta_{\ell}(B) - \delta'(\alpha_1 + \beta_1)$ 
for $\delta' \leq \delta$.
\end{itemize}
Suppose that $d > 0$.
Consider the first time of Case 1 
in which $\alpha'_1 + \beta'_1 \geq 0$ or $\ell+\delta = n$.
For the former case, 
it must hold 
$\alpha_1' + \beta_1' = \cdots = \alpha'_{n-\ell} + \beta'_{n-\ell} \geq 0 \geq \alpha'_{n-\ell+1} + \beta'_{n-\ell+1} \geq \cdots \geq \alpha'_{n} + \beta'_{n}$, and 
$(t^{\alpha'})P',Q'(t^{\beta'})$ is also optimal for MVMP$_\ell$.
This means that $\ell, ((t^{\alpha'})P',Q'(t^{\beta'}))$ is a required saddle-point.
For the latter case, 
$n, ((t^{\alpha'})P',Q'(t^{\beta'}))$ is a saddle-point.
If $d \leq 0$, then the initial solution with $\ell = 0$  is a saddle point. 
This completes the proof.

Finally, we point out the concavity property of $\varDelta_{\ell}$, 
which will be used later. 
\begin{Prop}\label{prop:concavity}
$\ell \mapsto \varDelta_{\ell}(B)$ is discretely concave:
\[
 \varDelta_{\ell+1}(B) + \varDelta_{\ell -1}(B) \leq 2\varDelta_{\ell}(B).
\]
\end{Prop}
Indeed, this follows from the above fact that
$\ell \mapsto \varDelta_{\ell+1}(B) - \varDelta_{\ell}(B)$ is monotone nonincreasing. 
It is known~\cite{HH_degdet,Oki} that
the function $(I, J) \mapsto \deg \Det B[I, J]$ is a 
\emph{valuated bimatroid}; see \cite[Section 5.2.5]{MurotaMatrix} for valuated bimatroids.
Then, this property can also be derived from the general property~\cite[Theorem 5.2.13]{MurotaMatrix} of valuated bimatroids.

\subsubsection{The {\bf Deg-Det} algorithm for $\varDelta_{\ell}(B)$}\label{subsub:DegSubDet}

The above proof of Theorem~\ref{thm:varDelta} is algorithmic,
and naturally yields the following algorithm for computing $\varDelta_{\ell}(B)$ for all $\ell$.
We define $\mindeg p/q$ for a nonzero rational $p/q$
as $\min\{i \mid a_i \neq 0\} - \deg q$, where $p= \sum_i a_it^i$.
\begin{description}
	\item[Algorithm: Deg-SubDet]
	\item[Input:] $B= \sum_{k =1}^m B_kx_k$, where  $d :=\max_{i,j} \deg B_{ij}$ and  $d_{0} := \min_{i,j: B_{ij} \neq 0} \mindeg B_{ij}$.
	\item[Output:] $\varDelta_{\ell}(B)$ for $\ell =1,2,\ldots,n$.
	\item[Initialization:] Let $\alpha := 0$,
	$\beta := - d{\bf 1}$, and $P=Q = I$.
	\item[0:] Compute $\ell := \ncrank ((t^{\alpha})PBQ(t^{\beta}))^{(0)}$, and
	output $\ell' d$ as $\varDelta_{\ell'}(B)$
	for $\ell' \in [\ell]$.
	\item[1:] Choose an optimal solution $S,T$,
	where $S((t^{\alpha})PBQ(t^{\beta}))^{(0)}T$ has an $r \times s$ zero submatrix in upper-left corner, and $2n- r-s  = \ell$.
      \item[2:] Let $\kappa > 0$ be the maximum integer such that
	$(t^{\kappa \11_{r}})S (t^{\alpha})P, Q(t^{\beta})T(t^{\kappa(\11_s-\11)})$ is feasible.
 Represent $(t^{\kappa \11_{r}})S (t^{\alpha})P = S' (t^{\alpha'})P'$ and $Q(t^{\beta})T(t^{\kappa(\11_s-\11)}) = Q' (t^{\beta'})T'$
 for $\alpha',\beta' \in \ZZ^n_{\downarrow}$ and biproper $S',T',P',Q'$.
 Let 
 $(t^{\alpha})P, Q (t^\beta) \leftarrow (t^{\alpha'})P',Q'(t^{\beta'})$.
% \[
 %\alpha \leftarrow \alpha',\ \beta \leftarrow \beta', \
 %P \leftarrow P',\ Q \leftarrow Q'.
 %\]
	\item[3:] Compute $\overline \ell := \ncrank ((t^{\alpha})PBQ (t^\beta))^{(0)}$:
	\begin{description}
		\item[3-1] Suppose that $\overline \ell > \ell$.
		Output $\varDelta_{\ell'}(B) := - \sum_{i=n-\ell'+1}^n\alpha_{i} - \sum_{j=n-\ell'+1}^n\beta_{j}$
		for $\ell' \in [\ell+1, \overline \ell]$.
		If $\overline{\ell} = n$, then stop.
		Otherwise, let $\ell \leftarrow \overline{\ell}$ and go to step 1.
		\item[3-2]  Suppose that $\overline{\ell} = \ell$.
		If $- \sum_{i=n-\ell+1}^n\alpha_{i} - \sum_{j=n-\ell+1}^n\beta_{j} < \ell d_{0}$, then output $\varDelta_{\ell'}(B) := -\infty$ for $\ell' \in [\ell, n]$, and stop.
		Otherwise, go to step 1.
	\end{description}
\end{description}

\begin{Thm}\label{thm:formal_primal_dual}
	{\bf Deg-SubDet} correctly outputs all $\varDelta_\ell(B)$ in $n(d-d_0)$ iterations.
%	\begin{itemize}
%		\item[(1)] 	The number of total iterations is bounded by $n(d- d_{0})$.
%		\item[(2)]  Until $\ncrank ((t^{\alpha})PBQ(t^{\beta}))^{(0)}$ changes (i.e., increases),
%		$r$ can change (i.e., decrease) at most $n$ times.
%	\end{itemize}
\end{Thm}
\begin{proof}
As mentioned in Section~\ref{subsub:building}, multiplication
of biproper matrices to feasible $(t^{\alpha})P,Q(t^{\beta})$
does not affect MVMP$_\ell$.
Therefore, several additional treatments
(e.g., to make the last $\ell \times \ell$ submatrix nc-nonsingular)
in the proof of Lemma~\ref{lem:(S,T)} and Theorem~\ref{thm:varDelta}
are not necessary.
Hence, it suffices to verify the correctness of the stopping criterion
in the case of $\varDelta_{\ell'}(B) = -\infty$ (step 3-2).
%\begin{Clm}[{\cite[Lemma 4.5]{HiraiIkeda}}]
%	For a (square) linear symbolic rational matrix $C$, it holds
%	$\deg \Det C = \frac{1}{g} \deg \det C^{\{g\}}$ for some integer $g > 0$.
%\end{Clm}
Suppose that $\deg \Det B[I,J] > -\infty$ for $I,J$ with $|I| = |J| = \ell$.
By Proposition~\ref{prop:blow-up_degDet}, for some integer $g > 0$, it holds
$\deg \Det B[I,J] =  \frac{1}{g} \deg \det B[I,J]^{\{g\}} \geq g \ell d_0/g \geq \ell d_0$.
This means that $\ell d_0$ is a finite lower bound of $\varDelta_\ell(B) > -\infty$.
%This implies (1).
%
%The property (2) follows from Lemma~\ref{lem:consecutive}.
\end{proof}

\subsection{Linear symbolic monomial matrices}

\subsubsection{Duality theorem}
We provide a sharpening of Theorem~\ref{thm:varDelta} for $A[c]$.
By a {\em complete flag} in $\KK^n$
we mean an $n+1$ tuple of vector subspaces 
$U_0,U_1,\ldots,U_n$ of $\KK^n$ such that
\begin{equation*}
\{0\} = U_0 < U_1 < \cdots < U_n = \KK^n.
\end{equation*}
\begin{Thm}\label{thm:minmax_subdet}
	Let $A = \sum_{k=1}^mA_kx_k$ be a matrix in \eqref{eqn:A} and let $c \in \ZZ^m$.
	Then
	$\varDelta_\ell(A[c])$ is equal to the optimal values of the following three (equivalent) problems:
		\begin{eqnarray*}
		\mbox{\rm (MVMP$^\KK_{\ell}$)} \quad {\rm Min.} && - \sum_{i=n-\ell+ 1}^n \alpha_i- \sum_{j=n-\ell + 1}^n \beta_j \\
		{\rm s.t.} && \deg ((t^{\alpha})PA_kt^{c_k} Q(t^{\beta}))_{ij} \leq 0 \quad (i,j \in [n], k \in [m]), \\
		&& \alpha, \beta \in \ZZ^n_{\downarrow}, P,Q \in GL_n (\KK).\\
				\mbox{\rm (MVMP$^{\rm F}_\ell$)}\quad {\rm Min.} && - \sum_{i=n- \ell+1}^n \alpha_i- \sum_{j=n -\ell +1}^n \beta_j \\
		{\rm s.t.} && \alpha_i + \beta_j \leq - c_k \quad
		(i,j \in [n],k \in [m]:A_k (U_i,V_j) \neq \{0\}), \\
		&& \alpha, \beta \in \ZZ^n_{\downarrow},\ \{U_i\},\{V_j\}:\mbox{complete flags in  $\KK^n$}. \\
			\mbox{\rm (MVMP$^{\rm A}_\ell$)}\quad {\rm Min.} &&  \sum_{i=1}^n \xi_i + \sum_{j=1}^n \eta_j  + \ell \gamma\\
		{\rm s.t.} && \xi_i + \eta_j + \gamma \geq c_k \quad
		(i,j \in [n],k \in [m]:A_k (u_i,v_j) \neq 0), \\
		&& \{u_1,u_2,\ldots,u_n\},  \{v_1,v_2,\ldots,v_n\}: \mbox{bases of $\KK^n$},\\
		&&  \xi,\eta \in \ZZ^n_+, \gamma \in \ZZ.
	\end{eqnarray*}
 In addition, $\varDelta_{\rm max}(A[c]))$ is equal to the optimal value of
\begin{eqnarray*}
   \mbox{\rm (MVMP$^{\rm F}_{\rm max}$)}  \quad {\rm Min.} && \sum_{i= 1}^n \xi_i  + \sum_{j= 1}^n  \eta_j  \\
    {\rm s.t.} && \xi_i + \eta_j \geq c_k \quad (i,j \in [n],k \in [m]: A_k(U_i,V_j)\neq \{0\}), \\
    && \xi,\eta \in \ZZ^{n}_{+\uparrow},\ \{U_i\},\{V_j\}: \mbox{complete flags in $\KK^n$}. 
\end{eqnarray*}
\end{Thm}
This theorem generalizes 
the duality theorem of Iwamasa~\cite{Iwamasa_degdet} for  
$\delta_{\ell}(A[c]) (= \Delta_{\ell}(A[c]))$ of $2 \times 2$-partitioned matrix $A$.
Note that $\varDelta_{\rm max}(A[c])$ is also given 
by unlisted problems MVMP$^{\KK}_{\rm max}$ and MVMP$^{\rm A}_{\rm max}$:
The former is obtained from MVMP$_{\rm max}$
by replacing $GL_n(\KK(t)^-)$
with $GL_n(\KK)$.
The latter is obtained from MVMP$^{\rm A}_\ell$ by omitting $\ell,\gamma$, which
is nothing but the LP-dual of the weighted bipartite matching problem if the bases are fixed.
Note also that these problems provide a good characterization for
$\varDelta_{\ell}(A[c])$ and 
$\varDelta_{\rm max}(A[c])$, if the bit-complexity 
of the bases of flags (or $P,Q$) is polynomially bounded.
%

%
%For the case of a $2 \times 2$-partitioned matrix $A$,
%Iwamasa~\cite{Iwamasa_degdet} developed a primal-dual combinatorial polynomial time %algorithm to compute $\delta_{\ell}(A[c]) = \Delta_{\ell}(A[c])$,
%and provided an algorithmic proof of the above duality.
%

This theorem is a consequence of
the following observation.
%Note that $S \in GL_n(\KK(t)^-)$ if and
%only if $S$ is expanded to $S = S^{(0)} + t^{-1}S^{(-1)} + \cdots$ with $S^{(0)} \in GL_n(\KK)$.
%
\begin{Lem}\label{lem:trick}
	For $\alpha,\beta \in \ZZ^n$ and $P,Q \in GL_n(\KK(t)^-)$,
	if $(t^{\alpha})P,Q(t^{\beta})$ is feasible for $A[c]$,
	then so is $(t^{\alpha})P^{(0)},Q^{(0)}(t^{\beta})$.
\end{Lem}
\begin{proof}
	Represent $P$ and $Q$ as $P = P^{(0)} + t^{-1}P'$ and $Q = Q^{(0)}+ t^{-1}Q'$, where
	$P^{(0)},Q^{(0)} \in GL_n(\KK)$ and $P',Q'$ are proper matrices.
	The feasibility of  $(t^{\alpha})P,Q(t^{\beta})$ is equivalent to that
	\[
	(P^{(0)}A_kQ^{(0)})_{ij} t^{\alpha_{i}+ \beta_j + c_k} +
	(P'A_kQ^{(0)} + P^{(0)}A_kQ')_{ij} t^{\alpha_{i}+ \beta_j + c_k-1} +
	(P'A_kQ')_{ij} t^{\alpha_{i}+ \beta_j + c_k  -2}
	\]
	is nonpositive degree for all $i,j,k$.
	Necessarily
	$(P^{(0)}A_kQ^{(0)})_{ij} t^{\alpha_{i}+ \beta_j + c_k}$ has nonpositive degree.
	This is feasibility of $(t^{\alpha})P^{(0)},Q^{(0)}(t^{\beta})$.
\end{proof}
\begin{proof}[Proof of Theorem~\ref{thm:minmax_subdet}]
By the above lemma, 
we can assume that solutions in MVMP$_\ell$ are the form of
$((t^{\alpha})P,Q(t^{\beta}))$ with $P,Q \in GL_n(\KK)$, which equals MVMP$^\KK_\ell$.
For such $((t^{\alpha})P,Q(t^{\beta}))$, 
define $U_i$ (resp. $V_j$) as the vector subspace spanned
by the first $i$ rows (resp. $j$ columns) of $P$ (resp. $Q$).
If $A_k(U_i,V_j) \neq \{0\}$,
then $PA_kQ[[i],[j]]$ has a nonzero entry $(PA_kQ)_{i'j'} \neq 0$
for $i' \leq i, j' \leq j$, and hence $\alpha_i + \beta_j \leq \alpha_{i'} + \beta_{j'} \leq - c_k$. Hence $\alpha, \beta, \{U_i\}, \{V_j\}$ is a feasible solution of MVMP$^{\rm F}_\ell$.

Let $\alpha, \beta, \{U_i\}, \{V_j\}$  be a feasible solution of MVMP$^{\rm F}_\ell$.
Choose bases $\{u_i\}, \{v_j\}$ of $\KK^n$
such that $u_1,u_2,\ldots,u_i$ span $U_i$ and $v_1,v_2,\ldots,v_j$ span $V_j$.
Let $P := (u_1\ u_2\ \cdots \ u_n)^{\top}$ and $Q := (v_1\ v_2\ \cdots \ v_n)$.
If $((t^{\alpha})PA_kQ(t^{\beta}))_{ij} = u_i^{\top}A_k v_j t^{\alpha_i + \beta_j}\neq 0$,
then $A_k(U_i,V_j) \neq \{0\}$, and hence $\alpha_i + \beta_j \leq - c_k$.
This means that $((t^{\alpha})P,Q(t^{\beta}))$ is feasible to MVMP$^{\KK}_\ell$.

MVMP$^{\rm A}_\ell$ is deduced as follows. As above,
consider a feasible solution $((t^{\alpha})P,Q(t^{\beta}))$ in MVMP$^{\KK}_\ell$.
We can assume that $\alpha_1 = \alpha_2 = \cdots = \alpha_{n-\ell} := \gamma_{1}$
and
$\beta_1 = \beta_2 = \cdots = \beta_{n-\ell} := \gamma_{2}$.
Then the objective value is written as $- \sum_{i=1}^n \alpha_i - \sum_{j=1}^n \beta_j + (n-\ell) (\gamma_1 + \gamma_2)$.
Define new variables $\xi,\eta,\gamma$ by
$\alpha_i = \gamma_1 - \xi_i$, $\beta_j = \gamma_2 - \eta_j$, and $\gamma := \gamma_1 + \gamma_2$.
Consider, as bases of $\KK^n$, row vectors $v_1,v_2,\ldots,v_n$ of $P$ and column vectors  $u_1,u_2,\ldots,u_n$ of $Q$.
By using them, we obtain MVMP$^{\rm A}_\ell$.

By applying Lemma~\ref{lem:trick} to MVMP$_{\rm max}$ similarly, 
we obtain MVMP$^{\rm F}_{\rm max}$. 
\end{proof}

We give useful observations for MVMPs in Theorem~\ref{thm:minmax_subdet}.
\begin{Lem}\label{lem:MVMP}
For MVMPs in Theorem~\ref{thm:minmax_subdet}, 
the following holds:
\begin{itemize}
\item[{\rm (1)}] 
	$\ZZ$ can be replaced by $\QQ$.
\item[{\rm (2)}] Suppose that each $A_k$ is symmetric or skew-symmetric.
	We can assume that $\{U_i\} = \{V_j\}$, $\alpha = \beta$, and $\xi = \eta$, with $\ZZ$ replaced by $\ZZ/2$ or $\QQ$. 
%	In {\rm MVMP}$^{\rm F}_\ast$, we may assume that  $\xi = \eta \in (1/2) \ZZ^n_{+ \uparrow}$ and $\{U_i\} = \{V_j\}$.
\end{itemize}
\end{Lem}
\begin{proof}
(1). Observe that 
MVMP$_\ell^{\rm F}$ is the integer program of 
the LP dual to 
the maximum weight $\ell$-matching problem in a bipartite graph, provided bases are fixed.
It is known (via the Hungarian method) 
that this LP admits an integral optimal solution.

(2). For two flags $\{U_i\},\{V_j\}$, there is 
a basis $\{ u_1,u_2,\ldots,u_n\}$ such 
that $U_i, V_j$ are spanned by some subsets of $\{ u_1,u_2,\ldots,u_n\}$.
Consequently,
we may assume $u_i  = v_i$ in MVMP$^{\rm A}_\ell$.
Since $A_k$ is symmetric or skew-symmetric, 
$A_k(u_i,u_j) \neq 0$ if and only if $A_k(u_j,u_i) \neq 0$.
In this case, it holds $\xi_i + \eta_j + \gamma \geq c_k$ and $\xi_j+ \eta_i + \gamma \geq c_k$,
and hence $(\xi_i+\eta_i)/2 + (\xi_j+\eta_j)/2 \geq c_k$.
Thus, we can replace $\xi,\eta$ by $(\xi+\eta)/2, (\xi+\eta)/2$. 
Translate it to MVMP$^{\rm F}_\ell$ to obtain the claim.
Similar for MVMP$^{\rm F}_{\rm max}$.
\end{proof}

\begin{Rem}[Spherical building]
Again, MVMP$_{\ell}^{\rm F}$ and MVMP$_{\ell}^{\rm A}$ have a natural interpretation in the theory of building.
The simplicial complex of all flags of vector subspaces is known as the {\em spherical building} 
for $GL_n(\KK)$.
MVMP$_{\ell}^{\rm F}$ and MVMP$_{\ell}^{\rm A}$ are 
viewed as optimization over (the Euclidean cone of) the geometric realization of the spherical building, 
where a point in this space is a conical combination of flags, as in the variables of MVMP$_{\ell}^{\rm F}$.
Further, bases of $\KK^n$ determine the family of subcomplexes of the building, called {\em apartments}. 
Apartments are isometric to $\RR^n$ 
and the building is their union. 
MVMP$_{\ell}^{\rm A}$ is viewed as
a multistage optimization of selecting an apartment and solving a linear program 
($=$ LP-dual of maximum-weight $\ell$-matching) on it. 
\end{Rem}

%See \cite{HH_improved_degdet} for a primal-dual type algorithm for $\deg \Det$.

\subsubsection{Polyhedral combinatorics on $\varDelta_\ell$}

Here we explain that $\varDelta_{\ell}(A[c])$ 
is given by linear optimization over some polytope.
Then the above duality can be viewed as a new type of linear programming duality.
Before that, we first consider $\delta_\ell(A[c])$.
For a polynomial $p(x_1,x_2,\ldots,x_m)= \sum_{u \in \ZZ_+^n} a_{u}x_1^{u_1}x_2^{u_2} \cdots x^{u_m}_m$,
let $\expvec p \subseteq \ZZ_+^m$ denote the set of
all integer vectors $u$ with  $a_{u} \neq 0$.
For $\ell \in [0,n]$, 
let ${\cal P}_\ell(A) \subseteq \RR^m_+$ be the polytope defined by
\begin{equation}\label{eqn:P_ell(A)}
	{\cal P}_\ell(A) :=  \Conv \bigcup_{I,J \in {[n] \choose \ell}} \expvec \det A[I,J],
\end{equation}
where $\det A[\emptyset, \emptyset] :=1$ and
${\cal P}_0(A)$ is the single point $\{(0, 0, \dots, 0)\}$.
Notice that ${\cal P}_n(A)$ 
is nothing but the {\em Newton polytope} of polynomial $\det A[c]$.
Further, let ${\cal P}(A)$ denote the polytope defined 
by taking the union over all subsets $I,J \subseteq [n]$ in (\ref{eqn:P_ell(A)}), 
or equivalently, ${\cal P}(A) := \Conv \bigcup_{\ell \in [0,n]} {\cal P}_\ell(A)$.
We call ${\cal P}(A)$ the {\em independent set polytope} for $A$.
For special symbolic matrices $A$ in Section~\ref{subsec:relationstoCO},
${\cal P}(A)$ can realize the following polytopes in combinatorial optimization.\footnote{For seeing (v) and (vi), we remark a general property 
that the maximum degree $\delta_{\rm max}$ 
of skew-symmetric rational matrix $B$ is always attained 
by a principal submatrix.
This fact can be proved e.g., via the valuated bimatroid property of $(I,J) \mapsto \deg \deg B[I,J]$: 
If $B[I,J]$ attains the maximum degree, then so does $B[J,I]$ by skew-symmetricity. By repeatedly applying the exchange axiom (VB-1), (VB-2)~\cite[Section 5.2.5]{MurotaMatrix} 
for $(I,J), (J,I)$, we eventually obtain a principal submatrix attaining the maximum degree.}.
\begin{itemize}
\item[(i)] The matching polytope for a bipartite graph.
\item[(ii)] The independent set polytope for a linear matroid.
\item[(iii)] The common independent set polytope for two linear matroids.
\item[(iv)] The matching polytope for a $2 \times 2$ partitioned matrix.
\item[(v)] Twice the matching polytope for a nonbipartite graph.
\item[(vi)] Twice the linear matroid matching polytope.
\end{itemize}

Observe that the maximum degrees $\delta_{\ell}$, $\delta_{\rm max}$ 
of minors of $A[c]$ are given
by linear optimizations over ${\cal P}_{\ell}(A)$:
\begin{equation}\label{eqn:delta_ell}
\delta_\ell (A[c]) = \max \{ c^{\top} u \mid u \in {\cal P}_\ell (A)\},\  \delta_{\rm max} (A[c])= \max \{ c^{\top} u \mid u \in {\cal P}(A)\}.
\end{equation}
Particularly, (weighted) Edmonds' problem can be understood as linear optimization over polytope ${\cal P}(A)$.

Hirai and Ikeda~\cite{HiraiIkeda} pointed out an analogous interpretation of $\deg \ncdet A[c]=\varDelta_n(A[c])$.
We here mention it for general $\ell$.
Recall the $d$-th blow up $A^{\{d\}}$ of $A$ in (\ref{eqn:blow-up}),
and consider its ordinary determinant $\det A^{\{d\}}$ that is a polynomial of variables
$x_{k,ij}$ for $k \in [m], i,j \in [n]$.
Then $\expvec \det A^{\{d\}}$ consists of $md^2$-dimensional integer vectors $z = (z_{k,ij})_{k \in [m],i,j \in [d]}$.
For such a vector $z= (z_{k,ij})_{k \in[m],i,j \in [d]}$, define
an $m$-dimensional vector  $\proj_d(z)\in \QQ^m$ by
\begin{equation}
	\proj_d (z)_k := \frac{1}{d} \sum_{i,j \in [d]} z_{k,ij} \quad (k \in [m]).
\end{equation}
%Also let $\proj_d (S) := \{ \proj_d(z) \mid z \in S \}$ for $S \subseteq \ZZ^{md^2}$.
For $\ell \in [0,n]$, 
define polytope ${\cal Q}_\ell (A) \subseteq \RR^m$ by
\begin{equation}\label{eqn:Q_ell(A)}
{\cal Q}_\ell(A) := \Conv \bigcup_{d=1,2,\ldots,} \bigcup_{I,J \subseteq {[n] \choose \ell}} \proj_d \expvec \det A[I,J]^{\{d\}}.
\end{equation}
Hirai and Ikeda~\cite{HiraiIkeda} call ${\cal Q}_n(A)$ 
the {\em nc-Newton polytope} for $A$.
%\begin{eqnarray}
%	{\cal Q}_\ell(A) &:=& \Conv \bigcup_{d=1,2,\ldots,} \bigcup_{I,J \subseteq {[n] \choose \ell}} \proj_d \expvec \det A[I,J]^{\{d\}}, \\
 %{\cal Q}(A) &:=& \Conv \bigcup_{d=1,2,\ldots,} \bigcup_{I,J \subseteq [n]} %\proj_d \expvec \det A[I,J]^{\{d\}},
%\end{eqnarray}
As in ${\cal P}(A)$, 
we define ${\cal Q}(A) (= \Conv \bigcup_{\ell \in [0,n]} {\cal Q}_\ell(A))$ by taking the union over all $I,J \subseteq [n]$ in (\ref{eqn:Q_ell(A)}).
We call ${\cal Q}(A)$ the {\em nc-indenpendent set polytope} for $A$. 
%An analogue of (\ref{eqn:delta_ell}) is the following.
\begin{Thm}\label{thm:polyhedral}
	\begin{itemize}
		\item[{\rm (1)}] $\varDelta_\ell(A[c]) = \max \{ c^{\top} u \mid u \in {\cal Q}_\ell(A)\}$ and  $\varDelta_{\rm max}(A[c]) = \max \{ c^{\top} u \mid u \in {\cal Q}(A)\}$.
		\item[{\rm (2)}] ${\cal Q}_\ell(A)$ and ${\cal Q}(A)$ 
  are integral polytopes in $[0,n]^m$.
		\item[{\rm (3)}] An integral vector $u$ maximizing $c^{\top} u$ over ${\cal Q}_{\ell}(A)$ is obtained in polynomial time.
	\end{itemize}
\end{Thm}
The meaning of polynomiality in (3) is the same as in Theorem~\ref{thm:main_IH}.
The case of $\ell = n$ is shown by \cite{HiraiIkeda}.
The proof is reduced to this case.
\begin{proof}
(1). $\varDelta_\ell(A[c]) = \max_{I,J \in {[n] \choose \ell}} \deg \Det A[c][I,J] =  \max_{I,J \in {[n] \choose \ell}} \max \{ c^{\top} u \mid u \in {\cal Q}_\ell(A[I,J])\} = \max  \{ c^{\top} u \mid u \in {\cal Q}_\ell(A)\}$.

(2). This follow from the fact that ${\cal Q}_{\ell}(A)$ is written as 
$\Conv \bigcup_{I,J \in  {[n] \choose \ell}} {\cal Q}_\ell(A[I,J])$.
Similar for ${\cal Q}(A)$.

(3). As in \cite[Section 4.3]{HiraiIkeda}, 
a maximizer $u$ can be identified from the values of $\varDelta_{\ell}(A[Kc + \epsilon])$
for a large constant $K> 0$ and $O(m)$ perturbed vectors $\epsilon$.
\end{proof}
In particular, the polyopes ${\cal Q}_\ell(A)$ and ${\cal Q}(A)$ are 
{\em integral} relaxations of ${\cal P}_\ell(A)$ and ${\cal P}(A)$, respectively:
$$
{\cal P}_\ell(A) \subseteq {\cal Q}_\ell(A),\quad {\cal P}(A) \subseteq {\cal Q}(A).
$$ 
Via relation $\delta_{\ell} = \varDelta_{\ell}$,
this relaxation is exact for the classes (i)--(iv). 
Although it is not exact for (v) and (vi),
we see in Section~\ref{sec:applications} that 
it equals the following well-known relaxations of (v) and (vi):
\begin{itemize}
\item[(v$^*$)] Twice the fractional matching polytope for a nonbipartite graph.
\item[(vi$^*$)] Twice the fractional linear matroid matching polytope.
\end{itemize}
%The polytopes ${\cal Q}_\ell(A)$ capture their cardinality-restricted version.

%We do not know the full combinatorial characterization of vertices of $Q_\ell(A)$.
%
%Combining Theorem~\ref{thm:minmax_subdet}, we obtain the following special type of LP duality:
%\begin{align*}
%\varDelta_\ell(A[c])\ = \ {\rm max.} &\quad c^{\top} u & =  \quad  {\rm min.} &\quad - \sum_{i=n - \ell + 1}^n \alpha_i- \sum_{j=n- \ell + 1}^n \beta_j  \nonumber\\
%	{\rm s.t.} &\quad  u \in Q_{\ell}(A)&  {\rm s.t.} & \quad \alpha_i + \beta_j \leq - c_k \quad
%	(i,j \in [n],k \in [m]:A_k (U_i,V_j) \neq \{0\}), \\
%	&&&\quad \alpha_1  \geq \alpha_2 \geq \cdots \geq \alpha_n,\ \beta_1 \geq \beta_2 \geq \cdots \geq \beta_n,   \\
%	&&&\quad U_1 \subset U_2 \subset \cdots \subset U_n,\ V_1 \subset V_2 \subset \cdots \subset V_n, \\
%	&&&\quad  \alpha, \beta \in \ZZ^n, U_i,V_j \subseteq \KK^n: \mbox{vector subspaces for $i,j \in [n]$}.
%	\nonumber
%\end{align*}
%In the case where

We note an expected relation between $\mathcal{Q}_\ell(A)$ and $\mathcal{Q}(A)$:
\begin{Prop}\label{prop:intersection}
    $\mathcal{Q}_\ell(A) = \{u \in \mathcal{Q}(A) \mid \mathbf{1}^\top u = \ell\}$.
\end{Prop}

\begin{proof}
    The inclusion $\subseteq$ is clear.
    We prove the reverse inclusion by showing $\max \{c^\top u \mid u \in \mathcal{Q}_\ell(A)\} \ge \max \{c^\top u \mid u \in \mathcal{Q}(A),\, \mathbf{1}^\top u = \ell\}$ for any $c \in \ZZ^m$, where the left-hand side is equal to $\varDelta_\ell(A[c])$ by Lemma~\ref{thm:polyhedral}~(1).
    Take a maximizer $u^*$ of the right-hand side.
    Since $u^* \in {\cal Q}(A) = \Conv \bigcup_{\ell=0}^n {\cal Q}_\ell(A)$, the vector $u^*$ can be decomposed as $u^* = \sum_{h=0}^n \lambda_h u_h$, where $\lambda_h \ge 0$, $\sum_{h=0}^n \lambda_h = 1$, and $u_h \in \mathcal{Q}_h(A)$ for each $h$.
    Then, it holds that $\ell = \mathbf{1}^\top u^* = \sum_{h=0}^n \lambda_h \mathbf{1}^\top u_h = \sum_{h=0}^n h\lambda_h$ and $c^\top u^* = \sum_{h=0}^n \lambda_h c^\top u_h \le \sum_{h=0}^n \lambda_h \varDelta_h(A[c])$.
By concavity of $h \mapsto  \varDelta_h(A[c])$ 
(Proposition~\ref{prop:concavity}), we have
$c^\top u^* \leq \sum_{h=0}^n \lambda_h \varDelta_h(A[c]) \le \varDelta_{\sum_{h=0}^n \lambda_h h}(A[c]) =\varDelta_\ell(A[c])$ as required.
%    By Proposition~\ref{prop:concavity}, there exists a concave extension $f: [0, n] \to \RR \cup \{-\infty\}$ of a sequence $(\Delta_0(A[c]), \Delta_1(A[c]), \dotsc, \Delta_n(A[c]))$, i.e., $f$ is a concave function with $f(h) = \Delta_h(A[c])$ for $h = 0, \dotsc, n$.
%    Therefore, by Jensen's inequality, we obtain $c^\top u^* \leq \sum_{h=0}^n \lambda_h f(h) \le f\big(\sum_{h=0}^n h\lambda_h \big) = f(\ell) = \Delta_\ell(A[c])$ as desired.
\end{proof}

\subsubsection{Algebraic Hungarian method}\label{subsub:Hungarian}
Here we specialize {\bf Deg-SubDet} for $A[c]$ to obtain a strongly polynomial time algorithm.
This is viewed as 
an algebraic generalization of the classical Hungarian method 
for weighted bipartite matching problems, 
and works as a unified primal-dual framework. 

We introduce notation to describe the algorithm. 
For $\alpha \in \ZZ^n_{\downarrow}$,
we define an ordered partition
$(I_1, I_2,\ldots,I_\mu)$
of $[n]$ such that for any indices $i,j \in [n]$ it holds
\begin{itemize}
\item $\alpha_i = \alpha_j$ if and only if both $i,j$ belong to $I_{a}$ for some $a \in [\mu]$, and
\item $\alpha_i > \alpha_j$ if $i \in I_a$, $j \in I_{a+1}$ for $a \in [\mu-1]$.
\end{itemize}

When applying {\bf Deg-Det} to $A[c]$, 
by Lemma~\ref{lem:trick}, 
we will always keep feasible $(t^{\alpha})P,Q(t^{\beta})$
such that $P,Q \in GL_n(\KK)$ and $p,q \in \ZZ_\downarrow^n$.
For such a feasible solution $(t^{\alpha})P,Q(t^{\beta})$,
we consider ordered partitions
$(I_1,I_2,\ldots,I_{\mu})$ from $p$ and
$(J_{1},J_2,\ldots,J_\nu)$ from $q$.
An optimal solution $S,T \in GL_n(\KK)$ of FR for  $((t^{\alpha})PAQ(t^{\beta}))^{(0)}$
is said to be {\em block-diagonal}
if $S[I_a,I_{a'}] = O$ and $T[J_b,J_{b'}] = O$ 
for all different indices $a\neq a'$, $b \neq b'$.
\begin{Lem}
For a feasible solution $(t^{\alpha})P,Q(t^{\beta})$ in MVMP,
there is a block diagonal optimal solution in FR for  $((t^{\alpha})PAQ(t^{\beta}))^{(0)}$.
\end{Lem}
\begin{proof}
Let $S = \pi_S U$, $T = M \pi_T$
be an optimal solution as in~(\ref{eqn:chosen}). 
Then $(t^{\11_r})S(t^{\alpha})P = (t^{\11_r}) \pi_S U (t^{\alpha})P 
= (t^{\11_r}) \pi_S (t^{\alpha})\tilde UP$ for biproper
$\tilde U := (t^{-\alpha}) U (t^{\alpha})$.
Similarly,
$Q(t^{\beta})T(t^{\11_s-\11}) = Q\tilde M(t^{\beta}) \pi_T (t^{\11s -\11})$ for
 biproper $\tilde M:= (t^{\beta}) M (t^{-\beta})$.
By Lemma~\ref{lem:trick},
$(t^{\11_r}) \pi_S (t^{\alpha})\tilde U^{(0)}P$, $Q\tilde M^{(0)} (t^{\beta}) \pi_T (t^{\11_s -\11})$
is also feasible.
Here $\tilde U^{(0)}$ and $\tilde M^{(0)}$ are
block-diagonal, and commute with $(t^{\alpha})$
and $(t^{\beta})$, respectively.
Thus $\tilde U^{(0)}((t^{\alpha})P A Q(t^{\beta}))^{(0)} \tilde M^{(0)}$
has a zero block of size $r+s$.
Necessarily $\tilde U^{(0)}, \tilde M^{(0)}$
is optimal and diagonal for FR.
\end{proof}
The proof provides a construction of such a solution:
Take an optimal solution $S= \pi_S U$, $T = M\pi_T$, and
replace each off-diagonal block $U,M$ by the zero matrix.
The resulting $U,M$ is a desired optimal solution.
Notice that the dominant-MVS property is preserved.
\begin{description}
	\item[Algorithm: Hungarian Deg-Det]
	\item[Input:] $A[c]= \sum_{k =1}^m A_kt^{c_k} x_k$, where $d:=  \max_{k \in [m]}c_k$.
	\item[Output:] $\varDelta_{\ell}(A[c])$ for $\ell =1,2,\ldots,n$.
	\item[Initialization:] Let $\alpha := 0$,
	$\beta := -d {\bf 1}$, and $P=Q = I$.
	\item[0:] Compute $\ell := \ncrank ((t^{\alpha})PA[c]Q(t^{\beta}))^{(0)}$, and
	output $\varDelta_{\ell'}(A[c]) := \ell' d$
	for $\ell' \in [\ell]$.
	\item[1:] Choose a dominant and block-diagonal optimal solution $S,T$ such that the row set $X$ and column set $Y$ of the zero submatrix of $S((t^{\alpha})PA[c]Q(t^{\beta}))^{(0)}T$
 with $2n- |X|-|Y|  = \ell$ are positioned as follows:
 \begin{itemize}
 \item $X \cap I_a$ and $Y \cap J_b$, if noempty, 
 constitute the first rows of $I_a$ and the first columns of $J_b$ for each index $a$ and $b$ of
 the partitions $(I_a)$ and $(J_b)$, respectively.
 \end{itemize}
      \item[2:] Let $\kappa_1>0$ be the maximum integer $\kappa$ such that
	$(t^{\alpha+\kappa\11_X}) SP, QT(t^{\beta+\kappa(\11_Y-\11)})$ is feasible, and let $\kappa_2>0$ be the maximum integer $\kappa$ such that $\alpha +\kappa\11_X \in \ZZ^n_\downarrow$
 and $\beta + \kappa(\11_Y-\11) \in \ZZ^n_{\downarrow}$.
 \item[3:] If $\kappa_1 = \infty$, then output $\varDelta_{\ell'}(A[c]) = -\infty$
 for $\ell' \in [\ell, n]$ and stop.
 Otherwise, let $\kappa := \min \{\kappa_1,\kappa_2\}$,
 and let $(t^{\alpha})P, Q (t^{\beta}) \leftarrow 
 (t^{\alpha+\kappa\11_X}) SP, QT(t^{\beta+\kappa(\11_Y-\11)})$. %as
 %\begin{equation*}
 %\alpha \leftarrow \alpha+\kappa\11_X,\
 %\beta \leftarrow \beta + \kappa(\11_Y-\11),\ P %\leftarrow SP,\  Q \leftarrow QT.
 %\end{equation*}
	\item[4:] Compute $\overline \ell := \ncrank ((t^{\alpha})PA[c]Q (t^\beta))^{(0)}$.
		If $\overline \ell > \ell$, then
		output $\varDelta_{\ell'}(A[c]) := - \sum_{i=n-\ell'+1}^n\alpha_{i} - \sum_{j=n-\ell'+1}^n\beta_{j}$
		for $\ell' \in [\ell+1,\overline \ell]$.
		In addition, If $\overline{\ell} = n$, then stop.
		Otherwise, let $\ell \leftarrow \overline{\ell}$ and go to step 1.
\end{description}

\begin{Thm}\label{thm:primal_dual}
{\bf Hungarian-Deg-Det}
correctly outputs all $\varDelta_{\ell}(A[c])$
in $O(r^*n^2)$ iterations, where $r^* := \ncrank A$.
\end{Thm}
The correctness is clear from Theorem~\ref{thm:formal_primal_dual}.
We show the iteration bound.
\begin{Lem}\label{lem:bar_ell=ell}
Suppose that $\overline{\ell} = \ell$ holds in step 4. Then we have the following:
\begin{itemize}
    \item[{\rm (1)}] $\kappa = \kappa_2$.
    \item[{\rm (2)}] Let $X'$ and $Y'$ denote the row and column sets, respectively,
of the zero submatrix
in the next iteration.
\begin{itemize}
    \item[{\rm (2-1)}] $X \neq X'$ or $Y \neq Y'$.
    \item[{\rm (2-2)}] There are surjections $\rho:X \to X'$ and $\varphi: [n] \setminus Y \to [n] \setminus Y'$ such that
    $\rho(i) \leq i$ for every $i \in X$
    and $\varphi(j) \geq j$ for every $j \in [n] \setminus Y$.
\end{itemize}
\end{itemize}
\end{Lem}

\begin{proof}
(1). Suppose for the contrary that $\kappa_1 < \kappa_2$.
Then, the new partitions $(I_a')$ and $(J_b')$ in the next iteration 
are refinements of the current partitions $(I_a)$ and $(J_b)$, respectively.
Namely, $(I_a')$ consists of $I_a \cap X$ and $I_a \setminus X$ for each $a$ (omitting empty sets) and $(J_b')$ is analogous.
Let $S',T'$ be an optimal solution of FR in the next iteration.  
Since $S'$ is block-diagonal relative to $(I_a')$ and $(I_a')$ is a refinement of $(I_a)$, 
$S'$ commutes with $(t^\alpha)$ as well as $(t^{\kappa \11_X})$.
Similarly, $T'$ commutes with $(t^{\kappa (\11_Y-\11)})$ and $(t^\beta)$.
So,
\begin{align*}
    &(t^{\kappa'{\bf 1}_{X'}})S'(t^{\kappa {\bf 1}_X})(t^{\alpha}) SPAQT (t^{\beta}) (t^{\kappa ({\bf 1}_Y-{\bf 1})}) T' (t^{\kappa' (\11_{Y'}-\11)}) \\
    &= (t^{\kappa'\11_{X'} + \kappa\11_X + \alpha}) S'SPAQTT' (t^{\beta + (\kappa \11_{Y}-\11) + \kappa'(\11_{Y'} - \11)})
\end{align*}
is proper.
By the same argument as Lemma~\ref{lem:consecutive}, we must have $X' = X$ and $Y' = Y$.
However, this contradicts the choice of $\kappa=\kappa_1$ as we could have taken a larger step size $\kappa' + \kappa$.

(2). From (1), we have $\kappa_1 \geq \kappa_2$.
In this case, we may 
assume that $(I_a')$ is not a refinement of $(I_a)$; the other case for $(J_b')$ is similar.
Then, there exists some index $a^*$ such that the first row of $I_{a^*}'$ is not contained in $X$.
Since $X'$ is chosen from the top rows in each block of $(I'_a)$, we must have $X' \neq X$.
On the other hand, $X \cap I_a'$ (if nonempty) constitutes the bottom rows of $I_a'$.
By Lemma~\ref{lem:consecutive}~(2) applied to each diagonal block of $S'$, 
there is an injection from $X' \cap I_a'$ to $X \cap I_a'$. 
Thus, there must exist a desired surjection $\rho$.
\end{proof}
%%%%%%%%%%%%%%%%%%%%%%%%%%%%%%
Now Theorem~\ref{thm:primal_dual} follows from:
\begin{Lem}
Within $O(n^2)$ iterations, it holds $\kappa_1 =\infty$
or $\overline{\ell} > \ell$.
\end{Lem}
\begin{proof}
Consider successive iterations
with $\overline{\ell} = \ell$ holding.
By Lemma~\ref{lem:bar_ell=ell}~(2), after $O(n^2)$ iterations,
$X$ and $Y$ in step 1 become the form of
$X = [r]$ and $Y=[s]$ for some $r,s$.
Then $\kappa_2 = \infty$.
Suppose that $\kappa_1 \neq \infty$.
Then $\kappa = \kappa_1 < \kappa_2$.
By Lemma~\ref{lem:bar_ell=ell}~(1),
it holds $\overline{\ell} > \ell$ in step 4.
\end{proof}
\begin{Rem}\label{rem:partition}
Suppose that $A$ is a partition matrix in (\ref{eqn:2x2}), 
and block-diagonal $P,Q$ in (\ref{eqn:block-diagonalized})
are used. Then, by considering the partition 
of each block, 
the bound $O(n^2)$ is improved to $O(\sum_i n_i^2 + \sum_j m_j^2)$.
\end{Rem}

\subsubsection{A primal-dual framework and its relation to previous algorithms}\label{subsub:primal-dual}

The {\bf Hungarian Deg-Det} algorithm 
may be viewed as a primal-dual algorithm if it keeps, as a primal solution,  
an $\ell \times \ell$ nc-nonsingular submatrix 
$((t^{\alpha})PA[c]Q(t^{\beta}))^{(0)}[I,J]$,
such that $I$ and $J$ 
take $\ell$ smallest values of $\alpha$ and $\beta$, respectively.
Then $\varDelta_{\ell}(A[c])$ is given by $- \sum_{i \in I} \alpha_i - \sum_{i \in J} \beta_j$. 
%This viewpoint leads to, in a unified way, previous algorithms as well as new variants for the problems in Section~\ref{subsec:relationstoCO}.
%

Suppose that  we have more ``functional'' primal solution $M$ 
satisfying:
\begin{itemize}
\item[(i)] $M$ certifies the nc-nonsingularity of
$((t^{\alpha})PA[c]Q(t^{\beta}))^{(0)}[I,J]$ 
before and after the update of $(t^{\alpha})P, Q(t^{\beta})$, and 
\item[(ii)] $M$ helps to find a dominant-MVS if $\overline{\ell} = \ell$ and
to find a larger nc-nonsingular submatrix including
$((t^{\alpha})PA[c]Q(t^{\beta}))^{(0)}[I,J]$ if $\overline{\ell} > \ell$. 
\end{itemize}
Let $\sigma_{\rm mvs}$ denote 
the time complexity for (i)
and let $\sigma_{\rm aug}$ denote   
the time complexity for (ii) that includes the update of $M$.
Additionally, let $\sigma_{\kappa}$ denote the time complexity of computing $\kappa$.
In this setting, {\bf Hungarian Deg-Det} runs in $O(n^2 (\sigma_{\rm mvs} + \sigma_{\kappa}) + r^* \sigma_{\rm aug})$ time.

For a matrix (\ref{eqn:bipartite}) of bipartite matching, 
{\bf Hungarian Deg-Det}
(with diagonal $P,Q$) 
becomes the classical Hungarian method, where 
a primal solution $M$ is just a matching in the bipartite graph, 
and the number $O(n^2)$ of dual update is improved to $O(n)$ (Remark~\ref{rem:partition}).
The optimality checking and a dominant-MVS finding are easily done from 
the residual graph of $M$. 
The augmentation procedure is edge-flip along the augmenting path.

For a matrix (\ref{eqn:A_intersection}) of linear matroid intersection,  
a primal solution $M$ is taken as a common independent set of two linear matroids.
Via Gaussian elimination, it specifies 
an $|M| \times |M|$ (nc-)nonsingular matrix and the residual graph 
in which a dominant-MVS or an augmenting path is easily obtained. 
The resulting algorithm matches 
a matrix formulation~\cite{FurueHirai} of
Frank's weight-splitting algorithm~\cite{Frank1981}.

For a $2 \times 2$ partitioned matrix (\ref{eqn:2x2}), {\bf Hungarian Deg-Det} yields 
a new primal-dual algorithm that is similar to but different from the one by Iwamasa~\cite{Iwamasa_degdet}.
A primal solution $M$ is a (multi-)subset of $2 \times 2$ blocks 
satisfying certain algebraic conditions, 
which is called a {\em matching} in \cite{HiraiIwamasaIPCO}, and 
certifies (nc-)nonsingularity of a submatrix meeting the blocks.
As $\alpha, \beta$ change,  
$((t^{\alpha})PA[c]Q(t^{\beta}))^{(0)}[I,J]$ 
(in suitable arrangement)
changes as  
\[
	\left(
	\begin{array}{cc}
		O & A_2^* \\
		A_1^* & \ast
	\end{array}
	\right) \to
 \left(
	\begin{array}{cc}
		O & A_2^* \\
		A_1^* & O
	\end{array}
	\right) \to 
 \left(
	\begin{array}{cc}
		\ast & A_2^* \\
		A_1^* & O
	\end{array}
	\right),
\]
where $A_1^*$ and  $A_2^*$ are square matrices.
See the last of the proof of Lemma~\ref{lem:(S,T)}~(2).
A matching $M$ is a subset of blocks 
meeting $A_1^*$, $A_2^*$, and remains 
a certificate of nonsigularity during the change.
For dominant-MVS finding and augmentation, 
one can use parts of the algorithm by Hirai and Iwamasa~\cite{HiraiIwamasaIPCO} 
for (nc-)rank computation of a $2 \times 2$ partitioned matrix. 
The former is done in $O(n)$ time
and the latter is done in $O(n^3)$ 
time\footnote{Although the paper~\cite{HiraiIwamasaIPCO} 
showed $O(n^4)$ running-time of the augmentation, Iwamasa~\cite{Iwamasa_degdet} pointed out an improved estimate $O(n^3)$}.
The number of dual updates is $O(n)$ 
(Remark~\ref{rem:partition}).
The total time complexity is $O(r^*n^3)$, which matches~\cite{Iwamasa_degdet}. 

For these examples, a primal solution $M$ corresponds to 
an (extreme) point 
in ${\cal P}(A) = {\cal Q}(A)$.
This suggests the design of a general primal-dual algorithm
that uses a point in ${\cal Q}(A)$ as a primal solution.
This is natural from the view of Theorem~\ref{thm:polyhedral}~(1).
%This issue, however, goes far beyond the scope of this paper.
The case of the matrix $A_{\cal H}$ in (\ref{eqn:linearmatroidmatching}),  where ${\cal P}(A_{\cal H}) \neq {\cal Q}(A_{\cal H})$, is a touchstone toward such an algorithm, and is studied in Section~\ref{subsub:frac_linear_algo}.   

\section{Applications}\label{sec:applications}
In this section, we present applications of our results to fractional matroid matching and the Brascamp-Lieb inequality.
\subsection{Fractional linear matroid matching}\label{subsec:frac_linear}
Let ${\cal H} = \{ H_1,H_2,\ldots,H_m\}$ be a collection of 2-dimensional subspaces in $\KK^n$.
A {\em fractional matroid matching} (Vande Vate~\cite{VandeVate92}) for ${\cal H}$ is a nonnegative vector $y \in \RR^m_+$ satisfying
\[
\sum_{k=1}^m y_k \dim H_k \cap X \leq \dim X \quad (X \leqslant \KK^n).
\]
In addition, if $2 \sum_{k=1}^my_k = n$, then it is said to be {\em perfect}.
The {\em fractional matroid matching polytope} $FMP({\cal H})$ for ${\cal H}$ is the polytope consisting of all fractional matroid matchings.
Suppose that for each $k \in [m]$, the subspace
$H_k$ is spanned by vectors $a_k,b_k \in \KK^n$.
As in (\ref{eqn:linearmatroidmatching}), 
define linear symbolic matrix $A_{\cal H} := \sum_{k=1}^m (a_k b_k^{\top} - b_k a_k^{\top}) x_k$.
As mentioned in Section~\ref{subsub:nonbipartite}, 
rank and nc-rank differ for this class of matrices.
However, Oki and Soma~\cite{OkiSoma} 
showed that $\ncrank A_{\cal H}$ still has an interesting interpretation---it is equal to twice the maximum of a fractional matroid matching:
\begin{Thm}[\cite{OkiSoma}]
	$\ncrank A_{\cal H} = 2 \max\{ {\bf 1}^{\top}y \mid y \in FMP({\cal H})\}$.
\end{Thm}
This also gives an interpretation to nc-rank
of matrix (\ref{eqn:Tutte}) as fractional matching number. 
We here establish a weighted generalization of this relation as follows:
\begin{Thm}\label{thm:frac_matroid_matching}
Let $c \in \ZZ^m$ and let $\ell \in [n]$. Then it holds
$$
	\varDelta_{\ell} (A_{\cal H}[c]) =  2 \max\{ c^{\top}y \mid y \in FMP({\cal H}), 2{\bf 1}^{\top}y = \ell\}.
$$ 
 In particular, it holds that ${\cal Q}(A_{\cal H}) = 2 FMP({\cal H})$ and ${\cal Q}_{n}(A_{\cal H}) = 2PFMP({\cal H})$.
\end{Thm}
The proof is given in Section~\ref{subsub:frac_linear_proof}.
%Before that, we explain algorithmic consequences.

\subsubsection{Algorithmic consequence}\label{subsub:frac_linear_algo}
By Theorem~\ref{thm:frac_matroid_matching}, 
the weighted fractional matching problem can be solved by the presented framework ({\bf Hungarian-Deg-Det}) for $\varDelta_{\ell}$, which requires $O(n^2)$ calls 
of a dominant-MVS computation subroutine.
The resulting algorithm is viewed as a variant of 
the algorithm by Gijswijt~and~Pap~\cite{GP13} (in the linear matroid setting). 
Their algorithm calls, in each iteration, 
the unweighted algorithm ({\rm CLV-algorithm})
by Chang, Llewellyn, and Vande Vate~\cite{CLV01a,CLV01b}
to find a {\em dominant $2$-cover} for the associated unweighted problem, and performs a dual update.
As shown in~\cite{OkiSoma},  
the dominant $2$-cover for 
a fractional matroid matching instance ${\cal H}'$
coincides with 
the dominant-MVS in FR for $A_{{\cal H}'}$.

A naive adaptation of \textbf{Deg-Det}
does not keep 
the leading term $((t^{\alpha})PA_{\cal H}Q(t^\beta))^{(0)}$ skew-symmetric, and is unable to use
the CLV-algorithm for finding the dominant-MVS.
This is possible by using,  via Lemma~\ref{lem:MVMP},  
a symmetric formulation of MVMP$_\ell^{\KK}$:
\begin{eqnarray*}
		{\rm Min.} && - 2\sum_{i=n-\ell+ 1}^n \alpha_i \\
		{\rm s.t.} && \deg ((t^{\alpha})PA_{\cal H}[c]P^{\top} (t^{\alpha}))_{ij} \leq 0 \quad (i,j \in [n]), \\
		&& \alpha \in \QQ^n_{\downarrow}, P\in GL_n (\KK).
\end{eqnarray*}
Then the leading term $((t^{\alpha})PA_{\cal H}P^{\top}
(t^\alpha))^{(0)}$ is skew-symmetric, and 
the dominant-MVS $(U,V)$ can be obtained by the CLV-algorithm.
It necessarily satisfies $U \geqslant V$. 
Let $S \in GL_n(\KK)$ include bases of $U$ and $V$ 
as row subsets $X$ and $Y$, respectively, where $X \supseteq Y$.
Then $S,S^{\top}$ is a  
dominant optimal solution of FR, and is taken to be block-diagonal relative to $\alpha$. 
Accordingly, the feasible solution $(t^{\alpha})P$ is updated in a symmetric way:
\begin{equation*}
(t^{\alpha})P \to (t^{\alpha+ \kappa((1/2){\bf 1}_{X}+ (1/2)({\bf 1}_{Y}-{\bf 1})})SP,
\end{equation*}
where step size $\kappa$ can be rational.
The proof arguments for \textbf{Hungarian Deg-Det} 
can be adapted for the symmetric version.
Actually, the proof relies on the partition structure induced by $\alpha,\beta$,
not the integrality of $\alpha,\beta$.
Further, one can verify that Lemmas~\ref{lem:(S,T)} and \ref{lem:consecutive} and the arguments of Section~\ref{subsub:proof} can be adapted for this setting.
Then, the resulting algorithm solves the weighted linear matroid matching in $O(n^2)$ calls of the CLV-algorithm, 
and has the same complexity as and looks quite similar to~\cite{GP13} (although we could not see the complete coincidence).

It is an interesting future research  
to incorporate the CLV-algorithm into  
a primal-dual framework in Section~\ref{subsub:primal-dual}, which will improve the time complexity greatly.
It is expected that an extreme fractional matching matching
will work as a primal solution and its characterization in~\cite{CLV01b} will help.

\subsubsection{Proof of Theorem~\ref{thm:frac_matroid_matching}}\label{subsub:frac_linear_proof}
By Proposition~\ref{prop:intersection}, it suffices to show ${\cal Q}(A_{\cal H}) = 2 FMP({\cal H})$.
We show an equivalent relation $\varDelta_{\rm max}(A[c]) = 2 \max \{c^{\top} y \mid y \in FMP({\cal H})\}$.
By Theorem~\ref{thm:minmax_subdet} and Lemma~\ref{lem:MVMP}, $\varDelta_{\rm max}(A[c])$ is given by
\begin{align}\label{eq:LP-FMP}
    \begin{split}
    {\rm Min.} \quad &  2 \sum_{i=1}^n \xi_i  \\
    {\rm s.t.} \quad &\xi_i +\xi_j  \geq c_k \quad (i,j \in [n],k \in [m]:(a_k b_k^{\top} - b_k a_k^{\top})(U_i,U_j) \neq \{0\}),  \\
    & \xi \in \QQ^n_{+\uparrow},\  \{U_i\}: \text{complete flags in $\KK^n$}. 
    \end{split}
\end{align}
We show that this problem reduces to the LP-dual of $\max \{2c^{\top} y \mid y \in FMP({\cal H})\}$.

For a vector subspace $U \subseteq \KK^n$, let $U^{\bot}$ 
denote the orthogonal complement of $U$ with 
respect to bilinear form $(x,y) \mapsto \sum_{i=1}^n x_iy_i$. 
\begin{Lem}
	For $i \leq j$,
	it holds $(a_k b_k^{\top} - b_k a_k^{\top})(U_i,U_j) \neq \{0\}$ if and only if $\dim H_k \cap 
 U_i^{\bot} \leq 1$ and $\dim H_k \cap U_j^{\bot} = 0$.
\end{Lem}
\begin{proof}
From the expression $a_k b_k^{\top} - b_k a_k^{\top} = (a_k\ b_k) {0\ 1 \choose -1\ 0} (a_k\ b_k)^{\top}$, we observe that $(a_k b_k^{\top} - b_k a_k^{\top})(U_i,U_j) = \{0\}$ if  
$(a_k\ b_k)^{\top} U_i$ and $(a_k\ b_k)^{\top} U_j$ are the same $1$-dimensional space.
Then,
 $(a_k b_k^{\top} - b_k a_k^{\top})(U_i,U_j) \neq \{0\}$ if and only if $\dim\, (a_k\ b_k)^{\top} U_i \geq 1$ 
 and $\dim\, (a_k\ b_k)^{\top} U_j = 2$. 
From the relation $\dim\, (a_k\ b_k)^{\top} U = 2 - \dim H_k \cap U^{\bot}$, we have the claim.
\end{proof}

For $i \in [n]$, define $\lambda_i \geq 0$ by
\[
\lambda_i := \xi_i - \xi_{i-1},
\]
where we let $\xi_0 := 0$.
\begin{Lem}
	For $k \in [m]$, the constraints $\xi_i +\xi_j  \geq c_k$ for $i,j \in [n]$ with 
 $(a_kb_k^{\top}- b_ka_k^{\top})(U_i,U_j) \neq \{0\}$ are written as a single constraint
	\[
	\sum_{i=1}^n \lambda_i \dim H_k \cap U_{i-1}^{\bot} \geq c_k.
	\]
\end{Lem}
\begin{proof}
Let $i^*$ and $j^*$ denote the first indices with $\dim H_k \cap U_i^{\bot} = 1$ and $\dim H_k \cap U_j^{\bot} = 0$, respectively.
By the previous lemma, for $i \leq j$, 
 $(a_kb_k^{\top}- b_ka_k^{\top})(U_i,U_j) \neq \{0\}$
 if and only if $i^* \leq i$ and $j^* \leq j$.
As $\xi \in \QQ_{\uparrow}^n$,
the constraints in question are written as one constraint $\xi_{i^*} + \xi_{j^*} \geq c_k$.
Since $\dim H_k \cap U_i^{\bot} = 2$ for $i < i^*$,
$\dim H_k \cap U_i^{\bot} = 1$ for $i^* \leq i < j^{*}$,
and $\dim H_k \cap U_i^{\bot} = 0$ for $j^* < i$. 
we have $\xi_{i^*} + \xi_{j^*} = \sum_{i=1}^{i^*} \lambda_i +  \sum_{j=1}^{j^*}  \lambda_j 
= \sum_{i=1}^{i^*} 2 \lambda_i + \sum_{i=i^*+1}^{j^*} \lambda_i 
= \sum_{i=1}^n \lambda_i \dim H_k \cap U_{i-1}^{\bot}$, as required.
\end{proof}

By using $\lambda$, the objective function is written as
\[
2 \sum_{i=1}^n (\lambda_1+ \lambda_{2} + \cdots + \lambda_i)
= 2 \sum_{i=1}^n (n-i+1) \lambda_i =2 \sum_{i=1}^n \lambda_i \dim U_{i-1}^{\bot}. 
\]
For $i \in [n]$, define $X_i := U_{i-1}^{\bot}$.
Then our problem is rewritten as
\begin{eqnarray*}
	{\rm Min.} && 2 \sum_{i=1}^{n} \lambda_i \dim X_i \\
	{\rm s.t.} && \sum_{i=1}^{n} \lambda_i \dim X_i \cap H_k  \geq c_k\quad (k \in [m]),\\
	&& \KK^n = X_1 > X_2 > \cdots > X_{n} > \{0\},\  \lambda \in \QQ^{n}_+.
\end{eqnarray*}
Consider the following relaxation, which is an infinite linear program:
\begin{align}\label{eqn:relax}
\begin{split}
{\rm Min.} \quad & 2 \sum_{X \in {\cal S}(\KK^n)} \lambda(X) \dim X \\
	{\rm s.t.} \quad & \sum_{X \in {\cal S}(\KK^n)} \lambda(X) \dim X \cap H_k  \geq c_k\quad (k \in [m]), \\
	\quad & \lambda: {\cal S}(\KK^n) \to \QQ_+, |\{X \in {\cal S}(\KK^n) \mid \lambda(X) > 0\}| < \infty, 
\end{split}
\end{align}
where ${\cal S}(\KK^n)$ denotes the family of all vector subspaces of $\KK^n$.

This relaxation is tight. To see this, we use the standard argument of {\em uncrossing}.
A function $f:{\cal S}(\KK^n) \to \QQ$ is called {\em submodular} if $f(X)+f(Y) \geq f(X \cap Y) + f(X+Y)$ for every $X,Y \in {\cal S}(\KK^n)$. In addition, it is called {\em strict} if the inequality is strict for every $X,Y$ with $X \not \leqslant Y$ and  $Y \not \leqslant X$.  
\begin{Lem}
	\begin{itemize}
		\item[{\rm (1)}] The function $X \mapsto - \dim X \cap H$ is submodular for any $H \leqslant \KK^n$.
		\item[{\rm (2)}] The function $X \mapsto \dim X \dim X^{\bot}$ is strict submodular.
	\end{itemize}
\end{Lem}
\begin{proof}
(1) follows from $\dim X \cap H + \dim Y \cap H = \dim X \cap Y \cap H + \dim (X \cap H)+ (Y \cap H) \leq \dim (X \cap Y) \cap H + \dim (X + Y) \cap H$.
(2) follows from:
For $X,Y \leqslant \KK^n$, we can choose basis $\{u_1,u_2,\ldots,u_n\}$ of $\KK^n$
such that $X = \vecspan \{u_i\}_{i \in I}$,
$Y = \vecspan \{u_i\}_{i \in J}$, 
$X+Y = \vecspan \{u_i\}_{i \in I \cup J}$,
and $X \cap Y = \vecspan \{u_i\}_{i \in I \cap J}$ for some $I,J \subseteq [n]$.
Then the strict submodularity reduces to
the well-known fact that $I \mapsto |I||[n] \setminus I|$ for $I \subseteq [n]$ is (usual)
strict submodular.
\end{proof}

Now consider a solution $\lambda$ of (\ref{eqn:relax}), where 
$K$ denotes the common denominator of nonzero rationals $\lambda(X)$. 
Suppose that there are $X,Y \in {\cal S}(\KK^n)$ such that
$\lambda (Y) > 0$, $\lambda (X) > 0$, $X \not \leqslant Y$, and $Y \not \leqslant X$.
Then, for $\epsilon := \min(\lambda(X),\lambda(Y))$, replace $\lambda(X) := \lambda(X) - \epsilon$,
$\lambda(Y) := \lambda(Y) - \epsilon$,
$\lambda(X+Y) := \lambda(X+Y)+ \epsilon$, and  $\lambda(X \cap Y) := \lambda(X \cap Y)+ \epsilon$.
By submodularity (1), the resulting $\lambda$ is feasible (with common denominator $K$).
By modularity of $X \mapsto \dim X$, the objective value does not change.
By strict submodularity (2),
the (nonnegative) quantity $\sum_{X} \lambda(X) \dim X \dim X^{\bot}$
strictly decreases at least by $1/K$.
By repeating this in finite many times,
we obtain a solution $\lambda$ keeping the objective value so that the nonzero support of $\lambda$ is a chain.
Then we obtain a solution of the original problem (by omitting $\lambda(\{0\})$).

Finally, consider the LP-dual of (\ref{eqn:relax}):
$$
{\rm Max.} \quad \sum_{k=1}^m c_k z_k \quad
	{\rm s.t.} \quad \sum_{k=1}^m z_k \dim X \cap H_k \leq 2 \dim X \ (X \leqslant \KK^n), \ z \in \QQ_+^m.
$$
The strong duality holds, since the number of inequalities appearing in this problem is finite.
For $y = z/2$, this is nothing but twice $\max\{ c^{\top}y \mid y \in FMP({\cal H})\}$.

\subsection{Brascamp-Lieb inequality}

The celebrated {\em Brascamp-Lieb inequality}~\cite{BrascampLieb,Lieb} states that:
Given surjective linear maps $B_j: \RR^n \to \RR^{n_j}$ and nonnegative reals 
$p_j \in \RR_+$ for $j=1,2,\ldots,m$ ({\em BL-datum}), the inequality 
\begin{equation}
\int_{x \in \RR^n} \prod_{j=1}^m f_j(B_jx)^{p_j}
{\rm d}x
\leq C_{\rm BL} \prod_{j=1}^m \left( \int_{x_j \in \RR^{n_j}} f_j(x_j) {\rm d}x_j \right)^{p_j} 
\end{equation}
holds for any nonnegative-valued measurable funtions $f_j: \RR^{n_j} \to \RR_+$ $(j=1,2,\ldots,m)$. 
The best constant (the {\em BL-constant}) $C_{\rm BL}$ is given by
\begin{equation}
C_{\rm BL} := \left( \sup_{X_j} \frac{\prod_{j=1}^m (\det X_j)^{p_j}}{\det \sum_{j=1}^m p_j B_j^{\top}X_j B_j} \right)^{\frac{1}{2}}  \in [0,\infty),
\end{equation}
where $X_j$ ranges over all positive definite matrices of size $n_j$.
Several important geometric inequalities, such as the H\"{o}lder inequality, can be deduced from the BL-inequality for special BL-data. 
The most familiar one is the Cauchy-Schwarz inequality, 
which corresponds to $p_1=p_2 =1/2, B_1=B_2=I$.
The interesting case of the BL-inequality is of $C_{\rm BL} < \infty$, which is characterized by Bennett et al.~\cite{BCCT08} as follows:
\begin{Thm}[\cite{BCCT08}]
Given BL-datum $B_j: \RR^n \to \RR^{n_j}$, $p_j \in \RR_+$ $(j=1,2,\ldots,m)$,
the BL-constant is finite if and only if 
the BL-datum satisfies 
 $n = \sum_{j=1}^{n}p_jn_j$  and
$\sum_{j=1}^m p_j \dim B_j V \geq \dim V$ for every $V \leqslant \RR^n$.
\end{Thm}
In particular, the set of all vectors $p = (p_j) \in \RR^m_+$
with finite BL-constant
forms a rational polytope, 
which is called the {\em BL-polytope} associated with $\{B_j\}_{j}$ 
and is denoted by $BLP(\{B_j\}_j)$. 

The computational complexity of the BL-polytope
has recently attracted attention in theoretical computer science~\cite{GGOW18}.
To the best of our knowledge, only a weak membership algorithm for the general BL-polytope is known~\cite{Burgisser2018a}.
More precisely, given a rational point $p$, the algorithm of \cite{Burgisser2018a} can distinguish whether $p$ in the BL-polytope or $p$ is $\eps$-far from the BL-polytope in time $\poly(N,\eps)$, where $N$ is the total bit complexity of the input BL datum.
It is open if one can improve the time complexity to $\poly(N, \log(1/\eps))$.

BL-polytopes are interesting from the combinatorial optimization point of view: It is known~\cite{Barthe1998} that 
if each $B_j$ is rank-1, the BL-polytope is equal to 
the base polytope of the linear matroid associated 
with the row vectors of $B_j$.
In this case, the (strong) membership of the BL-polytope
can be solved in polynomial time, via the separation-optimization equivalence~\cite{GLS} 
with the matroid greedy algorithm. 

A further interesting connection was discovered by
Franks, Soma, and Goemans~\cite{FranksSomaGoemans2022}.
They showed that
the BL-polytope for rank-2 matrices (the {\em rank-2 BL-polytope}) coincides with
the perfect fractional matroid matching polytope for the row spaces of these matrices.
\begin{Thm}[\cite{FranksSomaGoemans2022}]
Suppose that each $B_j$ is rank-2.
Then the BL-polytope $BLP(\{B_j\}_j)$ coincides with
the perfect fractional matching polytope $PFMP(\{H_j\}_j)$, 
where $H_j$ is the $2$-dimensional subspace in $\RR^n$ spanned by the row vectors of $B_j$.
\end{Thm}

Therefore, a polynomial-time linear optimization algorithm for fractional matroid matching problem
implies a polynomial complexity of the membership algorithm for 
rank-2 BL-polytopes.
As mentioned in Section~\ref{subsub:Hungarian}, such an algorithm was 
given by Gijswijt and Pap~\cite{GP13}.
However, as pointed out by Franks, Soma, and Goemans~\cite{FranksSomaGoemans2022}, 
a direct application of this algorithm 
can yield exponential bit-complexities in intermediate numbers 
when $\KK = \QQ$ (the setting of the BL-inequality). 
They utilized modified operator scaling to establish
a good characterization for the membership of rank-2 BL-polytopes:
\begin{Thm}[\cite{FranksSomaGoemans2022}]
The membership problem of rank-2 BL polytopes is in ${\rm NP} \cap {\rm coNP}$.
\end{Thm}

The coincidence between ${\cal Q}_n(A_{\cal H})$ and $PFMP({\cal H})$ established in the previous subsection enables us to apply our framework, particularly the modulo-$p$ reduction trick, 
for linear optimization over rank-2 BL polytopes. Thus we have:
\begin{Thm}
	The membership problem of rank-2 BL-polytopes is in {\rm P}.
\end{Thm}

\section*{Acknowledgments}
The first author was supported by JST PRESTO Grant Number JPMJPR192A and JSPS KAKENHI Grant Number JP21K19759.
The second author was supported by
JSPS KAKENHI Grant Numbers JP20K23323, JP20H05795, JP22K17854.
The third author was supported by
JST ERATO Grant Number JPMJER1903 and JSPS KAKENHI Grant Number JP22K17853.
The fourth author was supported by JSPS KAKENHI Grant Number JP19K20212.

\bibliographystyle{plain}
\bibliography{degdet_improved}

\end{document}